\pdfoutput=1
\documentclass[12pt]{amsart}
\usepackage{amsfonts}
\usepackage{amsmath}
\usepackage{amsthm}
\usepackage[abbrev]{amsrefs}

\usepackage{mathtools}
\usepackage{commath}

\usepackage{bm}
\usepackage[dvips]{graphicx}
\usepackage{caption}
\usepackage{hyperref}

\usepackage{color}

\usepackage{cancel}

\usepackage{bigints}

\numberwithin{equation}{section}

\allowdisplaybreaks

\newcommand{\pr}[1]{#1^{\prime}}
\newcommand{\X}[1]{\mathcal{X}_{#1}}
\newcommand{\iid}{\mathrm{i.i.d}}
\newcommand{\Var}{\mathrm{Var}}
\newcommand{\Cov}{\mathrm{Cov}}
\newcommand{\Poi}{\mathrm{Poi}}

\newcommand{\R}{\mathbb{R}}
\newcommand{\Y}{\mathcal{Y}}
\newcommand{\C}{\check{C}}
\newcommand{\Pn}{\mathcal{P}_n}
\newcommand{\remove}[1]{}
\newcommand{\one}{{\bf 1}}
\def\P{\mathbb{P}}
\def\E{\mathbb{E}}

\newcommand{\bx}{{\bf x}}
\newcommand{\by}{{\bf y}}

\newcommand{\bi}{{\bf i}}
\newcommand{\bj}{{\bf j}}

\newcommand{\bbr}{\mathbb R}
\newcommand{\bbn}{{\mathbb N}}

\newcommand{\B}{\mathcal B}
\newcommand{\Xio}{\mathcal X_{i_1}}
\newcommand{\Xit}{\mathcal X_{i_2}}
\newcommand{\Xiot}{\mathcal X_{i_1} \cup \mathcal X_{i_2}}
\newcommand{\grnt}{g_{r_n(t)}}
\newcommand{\grntij}{g_{r_n(t)}^{(i,j)}}

\newcommand{\gone}{g_{r_n(t_1)}^{(i_1, j_1)}}
\newcommand{\gtwo}{g_{r_n(t_2)}^{(i_2, j_2)}}
\newcommand{\G}{\mathcal G}
\newcommand{\V}{\mathcal V}
\newcommand{\Hk}{\mathcal H_k}

\newcommand{\etakrd}{\eta_{k, \mathbb R^d}}
\newcommand{\nukrd}{\nu_{k, \mathbb R^d}}
\newcommand{\Eonij}{E_{1,n}^{(i,\mathbf j)}}
\newcommand{\Etnij}{E_{2,n}^{(\mathbf i,\mathbf j)}}

\theoremstyle{plain}
\newtheorem{theorem}{Theorem}[section]

\newtheorem{proposition}[theorem]{Proposition}
\newtheorem{corollary}[theorem]{Corollary}
\theoremstyle{definition}
\newtheorem{definition}[theorem]{Definition}
\newtheorem{remark}[theorem]{Remark}

\begin{document}

\title[Limit theorems for Betti numbers]{Limit theorems for Process-level Betti numbers for Sparse, Critical, and Poisson regimes}
\author{Takashi Owada and Andrew Thomas}
\address{Department of Statistics\\
Purdue University \\
IN, 47907, USA}
\email{owada@purdue.edu \\
thoma186@purdue.edu}
\thanks{This research is partially supported by the NSF : Probability and Topology \#1811428}

\begin{abstract}
The objective of this study is to examine the asymptotic behavior of Betti numbers of \v{C}ech complexes treated as stochastic processes and formed from random points in the $d$-dimensional Euclidean space $\R^d$. We consider the case where the points of the \v{C}ech complex are generated by a Poisson process with intensity $nf$ for a probability density $f$. We look at the cases where the behavior of the connectivity radius of \v{C}ech complex causes simplices of dimension greater than $k+1$ to vanish in probability, the so-called sparse and Poisson regimes, as well when the connectivity radius is on the order of $n^{-1/d}$, the critical regime. We establish limit theorems in all of the aforementioned regimes, a central limit theorem for the sparse and critical regimes, and a Poisson limit theorem for the Poisson regime. When the connectivity radius of the \v{C}ech complex is $o(n^{-1/d})$, i.e., the sparse and Poisson regimes, we can decompose the limiting processes into a time-changed Brownian motion and a time-changed homogeneous Poisson process respectively. In the critical regime, the limiting process is a centered Gaussian process but has much more complicated representation, because the \v{C}ech complex becomes highly connected with many topological holes of any dimension. 
\end{abstract}

\maketitle

\section{Introduction}
It's easy enough to tell the difference between a donut and a mug on a sunny day---it's much harder in a hurricane. In a state in which there are innumerable ways in which one may classify objects, it may be useful to see that compared to a baseball---the mug with its handle and the donut with its hole---both have something for you to put your hand through. The point being is that in a veritable storm surge of noise, the ability to categorize objects by their most essential structure is an important start in learning the sum of their properties. The problem of analyzing data in the presence of noise has always been a nuisance. With the advent of the application of algebraic topology to probabilitistic structures, the ability to capture the most prominent of features of a space has never been closer at hand. These techniques and their corresponding theory typically fall under the umbrella of topological data analysis (TDA). 

A brief introduction into concepts of algebraic topology is needed before moving onward. Though our introduction here will be theoretically impoverished, it will nonetheless provide an intuition for some of the concepts discussed in this study. Those wishing for an introduction to algebraic topology for statistical ends should see \cite{carlsson2014topological, wasserman2016}. Treatments from a topological perspective for practitioners of all sorts can be seen in \cite{ghrist2014}, and a rigorous treatment can be seen in \cite{hatcher}. In many of the studies on TDA, especially those specific to random topology, the \emph{Betti number} has been a main focus as a good quantifier of topological complexity beyond simple connectivity. Given a topological space $X$ and an integer $k \geq 0$, the $k$th homology group $H_k(X)$ is the quotient group $\text{ker} \, \partial_k / \text{im}\, \partial_{k+1}$, where $\partial_k, \partial_{k+1}$ are boundary maps for $X$. More intuitively, $H_k(X)$ represents a class of topological invariants representing $k$-dimensional ``cycles" or ``holes" as the boundary of a $(k+1)$-dimensional body. The $k$th Betti number of $X$, denoted by $\beta_k(X)$, is defined as the rank of $H_k(X)$. Thus $\beta_k(X)$ captures, in essence, the number of $k$-dimensional cycles in $X$ (in the following we write ``$k$-cycle" for short). 
Having dispatched with this formalism, it is useful to know that $\beta_0(X)$ represents the number of connected components of $X$, $\beta_1(X)$ the number of ``closed loops'' in $X$ and $\beta_2(X)$ the number of ``voids''. For a manifold embedded in $\R^d$ these are features in one, two and three-dimensional subspaces respectively. Though it is the case that $\beta_k(X)$ is defined for all integers $k \geq 0$, in Figure~\ref{f:examples} above $\beta_k(X) \geq 0$ for $k \geq 3$. 

\begin{figure}[!t]
\begin{center}
\includegraphics[width=14cm]{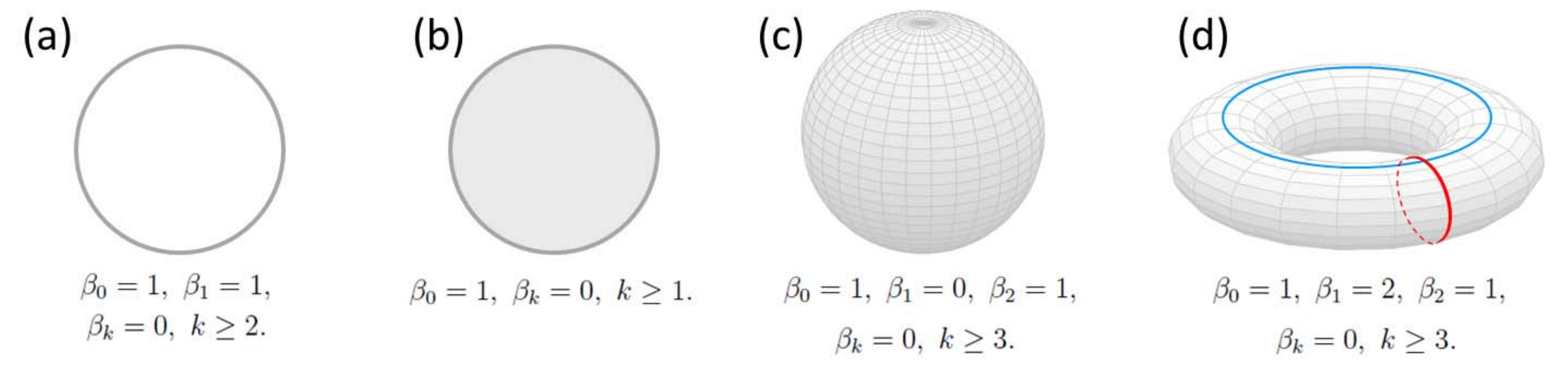}
\caption{{\footnotesize   The object in (a) is a 1-sphere, or a circle, i.e., $S^1 = \{x \in \R^2: | x | = 1\}$. The surface in (c) is a 2-sphere or $S^2$. Finally, (d) is a two-dimensional torus. Denoting the space corresponding to the torus as $X$, the blue and red cycles represent the generators of $H_1(X)$, and are regarded as non-equivalent cycles. Note that the torus is hollow, thus $\beta_2(X)$ = 1.}}
\label{f:examples}
\end{center}
\end{figure}

In recent years, there have been growing interests in the theory of random topology \cite{rgeom2011, bob2011, crackle, km, kmer, yogesh17}, exploring the probabilistic features of Betti numbers as well as related notions, for example, the number of critical
points of a certain distance function with a fixed Morse index. Additionally \cite{maximally} studied the maximal (persistent) $k$-cycles when an underlying distribution is a uniform Poisson process in the unit cube. Further, \cite{dec2014simp} investigated topology of a Poisson process on a $d$-dimensional torus. 
Those wishing to examine the properties of Betti numbers formed from points generated by a general stationary point process should consult \cite{yogadler2015, yogesh17}. An elegant summary on recent progress in the field is provided by  \cite{bobrowski:kahle:2018}. 
The topological objects in these studies are typically constructed from a \textit{geometric complex}. Among many choices of geometric complexes (see, e.g., \cite{ghrist2014}), the present paper deals with one of the most studied ones, a \textit{\v{C}ech complex}; see Figure \ref{f:cech}. 
\begin{definition} \label{cech}
If $t > 0$ and $\mathcal{X}$ is a collection of points in $\bbr^d$, the \v{C}ech complex $\C(\mathcal{X}, t)$ is defined as follows:
\begin{enumerate}
	\item The 0-simplices are the points in $\mathcal{X}$.
	\item A $k$-simplex $[x_{i_0}, \dots, x_{i_k}]$ is in $\C(\mathcal{X}, t)$ if $\bigcap_{j = 0}^k B(x_{i_j}; t/2) \neq \emptyset$, 
\end{enumerate}
where $B(x; r) = \{y \in \R^d: |x - y| < r\}$ is an open ball of radius $r$ around $x \in \R^d$.  
\end{definition}
\begin{figure}[!t]
\begin{center}
\includegraphics[width=9cm]{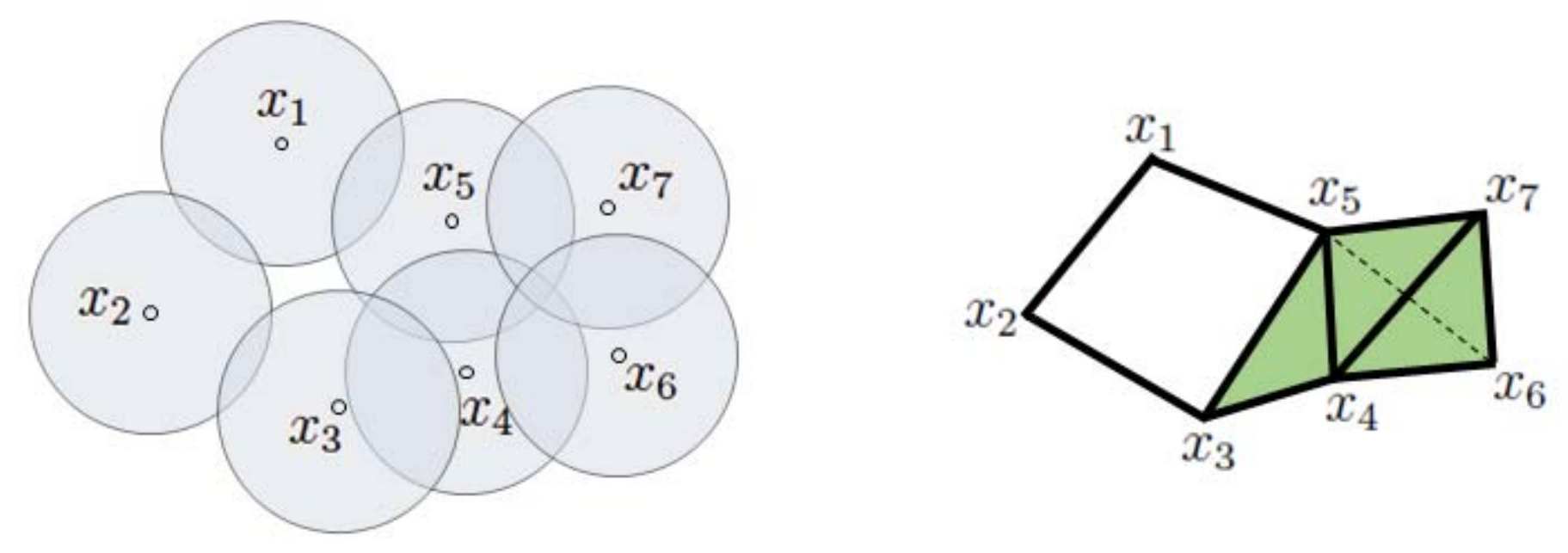}
\caption{{\footnotesize \v{C}ech complex $\check{C}(\mathcal X, t)$ with $\mathcal X=\{ x_1,\dots,x_7 \} \subset \bbr^2$. There are nine $1$-simplices with each adding a line segment joining a pair of the points. The $2$-simplex $[x_3,x_4,x_5]$ belongs to $\check{C}(\mathcal X, t)$, since the balls around these points have an non-empty intersection. The $3$-simplex $[x_4, x_5, x_6, x_7]$ represents a tetrahedron.  }}
\label{f:cech}
\end{center}
\end{figure}
One good reason for concentrating on the \v{C}ech complex is its topological equivalence to the union of balls $\bigcup_{y \in \mathcal X} B(y; t/2)$. A fundamental result known as the Nerve lemma (see, e.g., Theorem 10.7 of \cite{bjorner:1995}), asserts that the \v{C}ech complex and the union of balls are homotopy equivalent. In particular, they induce the same homology groups, that is for all $k \geq 0$
$$
H_k \big( \check{C}(\mathcal X, t) \big) \cong H_k \big(  \bigcup_{y \in \mathcal X} B(y; t/2) \big).
$$
The objective of the current paper is to investigate how the $k$th Betti number fluctuates as the sample size increases under the setup of \cite{km, bob2011, Bobrowski2015}. 
This setup necessarily dates back to the classical study on random geometric graphs as seen in the monograph \cite{penr}. This is due to the fact that a  \v{C}ech complex can be seen as a higher-dimensional analogue of a geometric graph. In fact, a geometric graph is actually a $1$-skeleton of a \v{C}ech complex. 
Let $\mathcal X_n$ be a set of random points on $\bbr^d$. Typically it represents $n$ i.i.d random points sampled from a probability density $f$ or a set of points taken from a Poisson process with intensity $nf$. Further, $r_n$ denotes a sequence of connectivity radii of a \v{C}ech complex (or a geometric graph). In this setting the behavior of $\check{C} (\mathcal X_n, r_n)$  is classified into several different regimes, depending on how $nr_n^d$ varies as $n\to\infty$. There is an intuitive meaning behind the quantity $nr_n^d$. It is actually the average number of points in a ball of radius $r_n$ around a point $x\in \bbr^d$, up to a proportionality constant. 

The first regime is that if $nr_n^d \to 0$ as $n\to\infty$, the complex is so sparse that many separate connected components are scattered throughout the space. This is called the \textit{sparse regime}. If the connectivity radii $r_n$ decays to $0$ more slowly, i.e., $nr_n^d \to \xi \in (0,\infty)$, then $\check{C}(\mathcal X_n, r_n)$ belongs to the \textit{critical regime}, in which the complex begins to be connected, forming much larger components with topological holes of various dimensions. Finally the case when $nr_n^d \to \infty$ is the \textit{dense regime}, for which the complex is highly connected with few topological holes. Detailed study of the Betti numbers has yielded a univariate central limit theorem for the sparse regime \cite{km, kmer} and for the critical regime \cite{yogesh17}. In addition \cite{km} has proven a Poisson convergence result of Betti numbers when $n^{k+2}r_n^{d(k+1)} \to \lambda \in (0,\infty)$ as $n\to\infty$, so that topological holes hardly ever occur. 

The main objective of this study is to generalize Betti numbers as a stochastic process and provide comprehensive results on limit theorems for the sparse, critical, and Poisson regimes. We shall consider the Betti  number of a \v{C}ech complex with radius $r_n(t) := s_nt$ :
\begin{equation}  \label{e:Betti.intro}
\beta_{k,n}(t) := \beta_k \big( \check{C} (\mathcal X_n, r_n(t)) \big),  \ \ \ t > 0. 
\end{equation}
Obviously \eqref{e:Betti.intro} gives a stochastic process in parameter $t$ with right continuous sample paths with left limits. 
With this functional setup, this paper reveals that when the \v{C}ech complex is relatively sparse, such as the sparse and Poisson regimes, the limiting process of $\beta_{k,n}(t)$ can be decomposed into the difference of well-known stochastic processes. Specifically, in the sparse regime we can decompose the limiting process into the difference of time-changed Brownian motions and in the Poisson regime we can decompose the limiting process as the difference of time-changed homogenous Poisson processes on the real half-line. In the critical regime however, the limiting process of $\beta_{k,n}(t)$ has much more complicated representation due to the emergence of connected components of larger size. In fact, the limiting process is denoted as the sum of infinitely many Gaussian processes with each representing connected components of size $i \geq k+2$ with $j$ topological holes. 

The motivation of reformulating Betti numbers as a stochastic process comes from an application to \textit{persistent homology}. Persistent homology is perhaps the most prominent and ubiquitous tool in TDA. Those needing a quick introduction should consult \cite{adler2010}. For surveys of applications of persistent homology see \cite{barcodes, carlsson2014topological, wasserman2016}. The first \cite{barcodes} is an essential and succinct overview. The second \cite{carlsson2014topological} gives a self-contained theoretical treatment of the topological and probabilistic aspects as well as detailed applications.The final one \cite{wasserman2016} gives an introduction to persistent homology and its applications from a statistical perspective. Theoretically rigorous treatment of persistence homology, especially the computational aspects, can be seen in \cite{edels2010, zomocarl}. 
Considering a family $\big( \check{C} (\mathcal X_n, r_n(t)), \, t > 0 \big)$ of \v{C}ech complexes and increasing radii $t$, the $k$th persistent homology provides a list of pairs (birth, death), representing the birth time (radius) at which a $k$-cycle is born and the death time (radius) at which it gets filled in and disappears. 
One of the typical applications of our results is the analysis on the sum of persistence barcodes, i.e. the sum of life lengths of all $k$-cycles up to time (radius) $t$, given by 
\begin{equation}  \label{e:sum.lifelength}
L_{k,n}(t) = \int_0^t \beta_{k,n}(s) \dif s,  \ \ \ t >0. 
\end{equation}
Of course, the limiting process of \eqref{e:sum.lifelength} is impossible to obtain from non-functional Betti numbers that do not involve parameter $t$. 
According to our results, however, it can be obtained as an integral of the limiting process of $\beta_{k,n}(t)$. Similar treatments of the stochastic process approach include \cite{owaSPB, owaFCLT}.

From the viewpoints of proof techniques we shall borrow ideas from \cite{penr, km, kmer} and apply sharper variance/covariance bounds than those given in \cite{kmer} for the analysis of the critical regime. Using these sharper bounds, the central limit theorem proven for the sparse regime no longer requires $s_n = o(n^{-1/d - \delta})$ for some $\delta > 0$ in the case that $n^{k+3}s_n^{d(k+2)}$ is bounded away from zero, as is assumed in \cite{kmer}. The argument for the Poisson regime uses a completely different technique based on \cite{dfsch16}. 

As a final remark, unlike \cite{penr, km, kmer} we do not consider points generated by a binomial process. Further studies would have to perform ``De-Poissonization'' as seen in section 2.5 of \cite{penr}. We have skipped these results not only for brevity but because they are highly technical and add little to the intuition behind our results. 

The structure of the paper is as follows. The second section details our setup and all the notation needed to appropriately and succinctly elucidate our results. The third section details the central limit theorem for the sparse regime, i.e. when we have $ns_n^d \to 0$ and $n^{k+2}s_n^{d(k+1)} \to \infty$. The fourth section is about the critical regime, in which $ns_n^d=1$, and Section \ref{s:Poisson} is dedicated to investigating the Poisson regime with $n^{k+2} s_n^{d(k+1)} = 1$. The major part of Section \ref{s:proofs} is devoted to proving limit theorems for the critical and Poisson regimes. The proof for the sparse regime can be obtained immediately via simple modification of the critical regime case. 

\section{Setup}

To begin, we start by defining some essential concepts towards proving the results in this paper. Due to the ease of proofs in the case of ``Poissonization''  we only look at point clouds generated by $\Pn$, a Poisson process on $\R^d$, $d\geq 2$. We take $\Pn$ to have the intensity measure $\lambda$ which we define as $\lambda(A) = n\int_A f(x)\dif{x}$ for all measurable $A$ in $\R^d$. In the previous definition $f$ is a probability density that is almost surely bounded and continuous with respect to Lebesgue measure. Throughout the paper, Lebesgue measure on $\R^{d(k+1)}$ is denoted by $m_k$ and for convenience we let $m := m_0$.

As an aside, we have a few definitions to mention before commencing. First, let $\lVert f \rVert_{\infty}$ be the essential supremum of the aforementioned $f$, which is finite as $f$ is almost surely bounded. Furthermore, define $\theta_d := m(B(0; 1))$ to be the volume of the unit ball in $\R^d$. The constant $C_{f,k}$ is mentioned frequently in the study and is defined as the quantity 
\[
C_{f,k} := \frac{1}{(k+2)!} \int_{\R^d} f(x)^{k+2} \dif{x}.
\]
Furthermore we let $\R_+ := [0, \infty)$ and $\mathbb{N}$ be the positive integers and $\mathbb{N}_0 := \mathbb{N} \cup \{0\}$---the non-negative integers, with $\one \bigl\{  \cdot \bigr\}$ denoting an indicator function.

It is useful to define the notion of a finite point cloud throughout the study. We let $\mathcal X_m:= \{X_1, X_2, \dots, X_m\}$ where $X_i$ are $\iid$ with density $f$ as mentioned before, though let it represent an arbitrary subset of $\R^d$ of cardinality $m$ as needed. Thus if $N_n$ is a Poisson random variable with parameter $n$, then we can represent the Poisson  process $\Pn$ as 
\[
\Pn(A) = \sum_{i = 1}^{N_n} \delta_{X_i}(A),
\]
for all measurable $A \subset \R^d$, with $\delta_x$ a Dirac measure at $x \in \R^d$.

With this definition in tow, we turn towards the study of Betti numbers. 

Fixing $1\leq k <d$, we define $h_t(x_1, \dots, x_{k+2})$, $x_i \in \bbr^d$, to be the indicator that $\C(\{x_1, x_2, \dots, x_{k+2}\}, t)$ contains an \emph{empty} $(k+1)$-simplex. This means that $\C(\{x_1, x_2, \dots, x_{k+2}\}, t)$ does not contain a $(k+1)$-simplex but does contain all possible $k$-simplices. 

With this in mind, we see that $h_t$ can be represented as
\[
h_t(x_1,\dots, x_{k+2}) = h_t^+(x_1,\dots,x_{k+2}) - h_t^-(x_1,\dots,x_{k+2}),
\]
where we define 
\begin{align*}
 h_t^+(x_1,\dots,x_{k+2}) &:= \prod_{i=1}^{k+2} \one \Bigl\{ \bigcap_{j=1, \, j \neq i}^{k+2} B(x_j; t/2)  \neq \emptyset  \Bigr\}, \\
 h_t^-(x_1,\dots,x_{k+2}) &:= \one \Bigl\{\, \bigcap_{j=1}^{k+2} B(x_j; t/2) \neq \emptyset \Bigr\}.
\end{align*}
It is important to note that $h_t^\pm$ is non-decreasing in $t$. That is, 
$$
h_s^\pm (x_1,\dots,x_{k+2}) \leq h_t^\pm (x_1,\dots,x_{k+2})
$$ 
for all $0 \leq s < t$ and $x_i \in \bbr^d$. 

Throughout the paper we interest ourselves in the $k$th Betti number for $\C(\Pn, r_n(t))$ where $r_n(t) := s_nt$. Recall that the nature of how $s_n$ decays to $0$ as $n \to \infty$ is the object of our study. We denote by $S_{k,n}(t)$ the number of empty $(k+1)$-simplex components of $\C(\Pn, r_n(t))$. In other words, $S_{k,n}(t)$ represents the number of connected components $C$ on $k+2$ points such that $\beta_k(C)=1$. More generally, for integers $i \geq k + 2$ and $j > 0$, we define $U_{i,j,n}(t)$ as the number of connected components $C$ of $\check{C}(\Pn, r_n(t))$ such that $|C| = i$ and $\beta_k(C) = j$. Then the \emph{k}th Betti number of $\check{C}(\Pn, r_n(t))$ can be represented as
\begin{equation} \label{e:Betti1}
\beta_{k, n}(t) =  \sum_{i \geq k+2} \sum_{j > 0} jU_{i,j}(t), \ \ \ t > 0.
\end{equation}
Since $S_{k,n}(t) = U_{k+2,1,n}(t)$ and one cannot form multiple empty $(k+1)$-simplices from $k+2$ points, \eqref{e:Betti1} can also be represented as 
\begin{equation}  \label{e:Betti2}
\beta_{k,n}(t) =   S_{k,n}(t) +  \sum_{i > k+2} \sum_{j > 0}  jU_{i,j,n}(t), \ \ \ t > 0. 
\end{equation}

In this setting it is instructive to introduce the following indicator functions to formalize these concepts for an arbitrary collection of points $\Y \subset \mathcal{X} \subset \R^d$: 
\begin{itemize} \itemsep0.5em
\item $J_{i,t}(\Y, \mathcal{X}) := \one \bigl\{ \check{C} (\Y, t) \text{ is a connected component of } \check{C}(\mathcal X, t)  \bigr\}\, \one \bigl\{ |\Y|=i \}$.
\item $b_{j,t}(\Y) := \one \bigl\{ \beta_k \bigl( \check{C} (\Y, t)  \bigr)=j \bigr\}$. 
\item $g_t^{(i,j)}(\Y, \mathcal{X}) := b_{j,t}(\Y) J_{i, t}(\Y, \mathcal{X})$.
\end{itemize}
In particular, denote
$$
g_t(\Y, \mathcal X) := g_t^{(k+2,1)}(\Y, \mathcal{X}) = b_{1,t}(\Y) J_{k+2, t}(\Y, \mathcal{X}) = h_t(\Y) J_{k+2, t}(\Y, \mathcal{X}).  
$$
Additionally, for $A \subset \mathbb{R}^d$,  let 
	\begin{itemize} \itemsep0.5em \vspace{0.5em}
	\item $h_{t,A}(\Y) := h_t(\Y)\one \{ \text{LMP} (\Y) \in A \}$,
	\item $g_{t, A}^{(i,j)}(\Y, \mathcal{X}) := g_t^{(i,j)}(\Y, \mathcal{X})\one \{ \text{LMP} (\Y) \in A \}$,
	\end{itemize}
where $\mathrm{LMP}(\Y)$ is the left-most point, in dictionary order, of the set $\Y$.

With the above indicators now available, it is clear that $S_{k,n}(t) = \sum_{\Y \subset \Pn} \grnt(\Y, \Pn)$ and $U_{i,j,n}(t) = \sum_{\Y \subset \Pn} \grntij (\Y, \Pn)$. As a final bit of notation, let 
\[
\beta_{k,n,A}(t) =  \sum_{i \geq k+2} \sum_{j > 0}  jU_{i,j,n,A}(t) = S_{k,n,A}(t) +  \sum_{i > k+2} \sum_{j > 0}  jU_{i,j,n,A}(t),
\]
where we require the left-most point of every subset $\Y$ to be an element of $A$ in the calculation of the $k$th Betti number. When brevity is paramount, we occasionally shorten $ \sum_{i > k+2} \sum_{j > 0} jU_{i,j,n}(t)$ to $R_{k,n}(t)$ and $ \sum_{i> k+2} \sum_{j > 0} jU_{i,j,n,A}(t)$ to $R_{k,n, A}(t)$ respectively. 

\section{Sparse regime}  \label{s:subcritical}
Throughout this section we assume that $ns_n^d \to 0$ and $\rho_n := n^{k+2} s_n^{d(k+1)} \to \infty$ as $n \to \infty$. The essence of the sparse regime is that \v{C}ech complexes are distributed sparsely with many separate connected components, because of a fast decay of $s_n$ as a result of $ns_n^d\to0$. Consequently, all $k$-cycles in the limit are supported exactly on $k+2$ points ($k+2$ is a ``minimum" number necessary to form a $k$-cycle). From a more analytic viewpoint, the behavior of the $k$th Betti number \eqref{e:Betti2} is completely determined by $S_{k,n}(t)$, whereas $R_{k,n}(t) = \beta_{k,n}(t) - S_{k,n}(t)$ is asymptotically negligible. 

The most relevant study to this section is \cite{km}, in which the central limit theorem for the sparse regime is discussed. We have extended \cite{km} (with the erratum paper \cite{kmer}) in twofold directions.  First, we develop the process-level central limit theorem for the sparse regime. This highlights the chief contribution of this paper. Whereas \cite{km, kmer}, as well as \cite{yogesh17} in the ensuing section, treat the ``static" topology of random \v{C}ech complexes (i.e., no time parameter $t$ involved), the main focus of this paper is ``dynamic'' topology of the same complex, treating Betti numbers as a stochastic process. Second, our central limit theorem is for the entirety of the sparse regime, without requiring that $s_n =  o(n^{-1/d  -\delta})$ for some $\delta > 0$ as assumed in \cite{kmer}. 

Before presenting the main result we define the limiting stochastic process
\begin{equation}  \label{e:limit.subcritical}
\mathcal{G}_k(t) := \int_{\R^{d(k+1)}} h_t(0, \mathbf{y}) \, G_k(\dif{\mathbf{y}}),
\end{equation}
where $G_k$ is a Gaussian random measure such that $G_k(A) \sim \mathcal N(0, C_{f, k}m_{k}(A))$ for all measurable $A$ in $\R^{d(k+1)}$. Furthermore, for $A_1, \dots, A_m$ disjoint, $G_k(A_1), \dots, G_k(A_m)$ are independent. As defined, $\G_k(t)$ depends on the indicator $h_t$, meaning that due to sparsity of the \v{C}ech complex in this regime, the $k$-cycles affecting $\G_k(t)$ must be always formed by connected components on $k+2$ points (i.e., components of the smallest size). 

The significance of the characterization of the process at \eqref{e:limit.subcritical} is that if we define 
\[
\mathcal{G}^{\pm}_k(t) := \int_{\R^{d(k+1)}} h^{\pm}_t(0, \by) \, G_k(\dif{\by}), 
\]
then $\mathcal{G}^{\pm}_k(t)$ becomes a time-changed Brownian motion; see Proposition \ref{P:gauproc} below.  Hence $\G_k(t) = \mathcal{G}^{+}_k(t) - \mathcal{G}^{-}_k(t)$ 
is a difference of two dependent time-changed Brownian motions, where dependence is due to the same Gaussian random measure $G_k$ shared by $\G_k^+(t)$ and $\G_k^-(t)$. Those wishing to examine this characterization in more detail should refer to \cite{owaFCLT}. For example, it is proven in \cite{owaFCLT} that the process $\G_k(t)$ is self-similar with exponent $H=d(k-1)/2$ and is H\"older continuous of any order in $[0,1/2)$. 

\begin{proposition}\label{P:gauproc}
The process $\mathcal{G}^{\pm}_k(t)$ can be expressed as 
\[
(\mathcal{G}^{\pm}_k(t), t \geq 0) \overset{d}{=} \Big(B(C_{f,k}m_k(D_1^{\pm})t^{d(k+1)}), t \geq 0\Big), 
\]
where $B$ is a standard Brownian motion and $D_t^{\pm} = \{\by \in \R^{d(k+1)}: h^{\pm}_t(0, \by) =1\}$.
\end{proposition}
\begin{proof}
We prove only the result for $\mathcal{G}^{+}_k$, as the proof for $\mathcal{G}^{-}_k$ is the same. It is elementary to show that $\mathcal{G}^{+}_k(t)$ has mean zero. Thus, it only remains to demonstrate the covariance result. Since $h_t^+$ is non-decreasing in $t$, we have $D_{t_1}^+ \subset D_{t_2}^+$ for $0 \leq t_1 \leq t_2$; therefore,  
\begin{align*}
\E\bigl[\mathcal{G}^{+}_k(t_1)\mathcal{G}^{+}_k(t_2)\bigr] &= \E\bigl[G_k(D_{t_1}^+) G_k(D_{t_2}^+)] = \E\bigl[G_k(D_{t_1}^+)^2] \\
&= C_{f,k} m_k(D_{t_1}^+) = C_{f,k}m_k(D_1^{+})t_1^{d(k+1)}. 
\end{align*}
\end{proof}
Our main result can be seen below. The proof is briefly presented in Section \ref{s:proof.subcritical} as a straightforward variant of the proof for the critical regime. For the proof we need to examine the asymptotic growth rate of expectations and covariances of $\beta_{k,n}(t)$. The detailed results are presented in Proposition \ref{P:covsparse}, where it is seen that the expectation and covariance both grow at the rate  $\rho_n$.

\begin{theorem}\label{T:sparse}
Suppose that $ns_n^d \to 0$ and $\rho_n = n^{k+2}s_n^{d(k+1)} \to \infty$. Assume that $f$ is an almost everywhere bounded and continuous density function. Then, we have the following weak convergence in the finite dimensional sense, namely
\begin{equation*}  
\rho_n^{-1/2} \bigl( \beta_{k,n}(t) - \E[\beta_{k,n}(t)] \bigr) \overset{fidi}{\Rightarrow} \mathcal{G}_k(t),
\end{equation*}
meaning that for every $m \in \mathbb{N}$ and $0 < t_1 < t_2 < \dots < t_m < \infty$ we have
\begin{equation*}
\rho_n^{-1/2}\Bigl( \beta_{k,n}(t_i) - \E[\beta_{k,n}(t_i)], \, i=1,\dots,m \Bigr) \Rightarrow \bigl( \mathcal{G}_k(t_i),\, i=1,\dots,m \bigr)
\end{equation*}
weakly in $\R^m$. 
\end{theorem}

\section{Critical regime}

We now expand on the results of \cite{yogesh17} by offering an explicit limit of appropriately scaled moments and a central limit theorem for $\beta_{k,n}(t)$. In the critical regime, the connectivity radius $s_n$ is 
defined to be $s_n = n^{-1/d}$. This sequence decays more slowly than that in the previous section; hence,  \v{C}ech complexes become highly connected with many topological holes of any dimension $k< d$. More analytically, all terms in the sum \eqref{e:Betti1} contribute to the $k$th Betti number, unlike in the sparse regime. This implies that the $k$-cycles in the limit could be supported not only on $k+2$ points but also on $i$ points for all possible $i > k+2$. 

As a related work, \cite{yogesh17} also established a central limit theorem for the critical regime (though \cite{yogesh17} referred to it as the ``thermodynamic'' regime). There are two key differences between that paper and ours. The first is that the Poisson process they consider is stationary with unit intensity, restricted to a set $B_n$ such that $m(B_n) = n$. The second difference between the two, and equivalent to the contrast indicated in the sparse regime, is again that \cite{yogesh17} treats the static topology of random \v{C}ech complexes whereas we treat the dynamic topology. As a consequence, while the weak limit in \cite{yogesh17} is a simple Gaussian distribution with unknown variance, our limit is a Gaussian process having structure similar to that of the Betti number \eqref{e:Betti1}.

We now define the limiting Gaussian process $\mathcal H_k(t)$
\begin{equation}  \label{e:Gaussian.series.representation}
\mathcal H_k(t)= \sum_{i \geq k+2} \sum_{j>0} j \mathcal H_k^{(i,j)}(t), \ \ \ t > 0,  
\end{equation}
where $\big( \mathcal H_k^{(i,j)}, \, i \geq k+2, j>0 \big)$ is a family of centered Gaussian processes with inter-process dependence between $\mathcal H_k^{(i_1,j_1)}$ and $\mathcal H_k^{(i_2,j_2)}$ determined by  
\begin{equation}  \label{e:cov.critical1}
\text{Cov} \big( \Hk^{(i_1,j_1)}(t_1), \Hk^{(i_2,j_2)}(t_2) \big) = \frac{1}{i_1!}\, \etakrd^{(i_1, j_1, j_2)}(t_1, t_2) \delta_{i_1, i_2} + \frac{1}{i_1! i_2!}\, \nukrd^{(i_1, i_2, j_1, j_2)} (t_1, t_2). 
\end{equation}
Here $\delta_{i_1, i_2}$ is the Kronecker delta, and the functions $\etakrd^{(i_1, j_1, j_2)}$, $\nukrd^{(i_1, i_2, j_1, j_2)}$ are explicitly defined during the proof of the main theorem (see \eqref{e:eta.k} and \eqref{e:nu.k}). From \eqref{e:cov.critical1}, the covariance of $\mathcal H_k^{(i,j)}$ is given by 
\begin{equation*}  \label{e:cov.critical2}
\text{Cov} \big( \Hk^{(i,j)}(t_1), \Hk^{(i,j)}(t_2) \big) = \frac{1}{i!}\, \etakrd^{(i, j, j)}(t_1, t_2) + \frac{1}{(i!)^2}\, \nukrd^{(i, i, j, j)} (t_1, t_2). 
\end{equation*}

The main point here is that the Betti number \eqref{e:Betti1} and the limit \eqref{e:Gaussian.series.representation} are represented in a very similar fashion. In fact, the process $U_{i,j,n}(t)$ in \eqref{e:Betti1} and $\Hk^{(i,j)}(t)$ in \eqref{e:Gaussian.series.representation} both capture the spatial distribution of connected components $C$ with $|C| = i$ and $\beta_k(C) = j$. 
In particular, $\Hk^{(k+2,1)}(t)$ represents the distribution of components $C$ on $k+2$ points with $\beta_k(C)=1$ (i.e., components of the smallest size) as does $\G_k(t)$ in the sparse regime. In the present regime however, many of the Gaussian processes in \eqref{e:Gaussian.series.representation} beyond $\Hk^{(k+2,1)}(t)$, do contribute to the limit. 

As a bit of a technical remark, note that for every $i \geq k+2$, there exists $j_0 > 0$ such that $b_{j,t}(\bx)=0$ for all $j \geq j_0$, $t>0$, and $\bx \in \bbr^{di}$. In this case, 
\[
\etakrd^{(i,j,j)}(t,t) = \nukrd^{(i,i,j,j)}(t,t) = 0,
\] and thus $\Hk^{(i,j)}$ becomes an identically zero process. For example, $\Hk^{(k+2,j)} \equiv 0$ for all $j \geq 2$, since one cannot create multiple $k$-cycles from $k+2$ points. 

In the result below  we let $ns_n^d = 1$, though we could easily suppose that $ns_n^d \to 1$ as $n \to \infty$. All proofs are collected in Section \ref{s:proof.critical}. Our proof is highly analytic in nature, borrowing techniques from \cite{penr} and \cite{km}, whereas the proof of \cite{yogesh17} relies more on topological nature of the objects. In particular we will see that the growth rate of the expectation and variance of $\beta_{k,n}(t)$ is of order $n$---see Proposition \ref{P:covcritical}. This indicates that the scaling constant for the central limit theorem must be of order $n^{1/2}$. 
\begin{theorem}\label{T:cclt}
Suppose that $ns_n^d = 1$ and $f$ is an almost everywhere bounded and continuous density function. If $0 < t_1 < t_2 < \dots < t_m < (e \lVert f \rVert_{\infty} \theta_d)^{-1/d}$, and $\Hk(t)$ is the centered Gaussian process defined above, then we have the following weak convergence in the finite dimensional sense, namely
\[
n^{-1/2}\big(\beta^{}_{k,n}(t) - \E[\beta_{k,n}(t)]\big) \overset{fidi}{\Rightarrow} \mathcal{H}_k(t). 
\]
This means that for every $m \in \mathbb{N}$ we have 
\begin{equation*} \label{e:crit.beta}
n^{-1/2}\bigl( \beta_{k,n}(t_i) - \E[\beta_{k,n}(t_i)], \, i=1,\dots,m \bigr) \Rightarrow \bigl(\mathcal{H}_k(t_i), \,  i=1,\dots,m\bigr),
\end{equation*} 
weakly in $\R^m$. 
\end{theorem}

\begin{remark}
Although Theorem \ref{T:cclt} imposes a restriction on the range of $t_i$'s, we conjecture that the limit theorem holds without such restrictions. Indeed in the case of  
the ``truncated" Betti number 
$$
\beta_{k,n}^{(M)}(t) = \sum_{i=k+2}^M \sum_{j >0} j U_{i,j,n}(t), \ \ M \in \bbn, 
$$
which itself is useful for the approximation arguments in our proof, the central limit theorem does hold for every $t>0$.  
\end{remark}
Before concluding this section we shall exploit Theorem 4.6 in \cite{yogesh17} and present the strong law of large numbers of $\beta_{k,n}(t)$. The proof is given at the end of Section \ref{s:proof.critical}. 
\begin{corollary}  \label{c:SLLN}
Under the condition of Theorem \ref{T:cclt}, we assume moreover that $f$ has a compact, convex support such that $\inf_{x \in \text{supp} (f)}f(x)>0$. Then 
we have, as $n\to\infty$, 
$$
\frac{\beta_{k,n}(t)}{n} \to \sum_{i=k+2}^\infty \sum_{j>0} \frac{j}{i!}\, \etakrd^{(i,j,j)}(t,t), \ \text{a.s.}
$$
\end{corollary}
\section{Poisson regime}  \label{s:Poisson}
Before concluding this paper we shall explore the random topology of \v{C}ech complexes when the complex is even more sparse than that in Section \ref{s:subcritical}, so that $k$-cycles hardly ever occur. Then, the $k$th Betti number no longer follows a central limit theorem. Nevertheless, it does obey a Poisson limit theorem.  
In terms of the connectivity radii, we assume $\rho_n = n^{k+2}s_n^{d(k+1)} = 1$, equivalently, $s_n = n^{-(k+2)/d(k+1)}$, so that $s_n$ converges to $0$ more rapidly than in the sparse regime. 

For the  definition of a ``Poissonian" type limiting process, we let $M_k$ be a Poisson random measure with mean measure $C_{f,k} m_k$. Namely it is defined by 
\[ M_k(A) \sim \Poi(C_{f,k} m_{k}(A))\]
for all measurable $A$ in $\R^{d(k+1)}$. Further, if $A_1,\dots,A_m$ are disjoint, $M_k(A_1), \dots, M_k(A_m)$ are independent. We are now ready to define the stochastic process
\[
\mathcal{V}_k(t) = \int_{\R^{d(k+1)}} h_t(0, \mathbf{y}) \, M_k(\dif{\mathbf{y}}), 
\]
which appears below as a weak limit in the main theorem. What is interesting about this is that if we define
\[
\mathcal{V}_k^{\pm}(t) := \int_{\R^{d(k+1)}} h_t^{\pm}(0, \mathbf{y}) \, M_k(\dif{\mathbf{y}}),
\]
then $\mathcal{V}_k(t) = \mathcal{V}_k^+(t) - \mathcal{V}_k^-(t)$ is the difference of two dependent (time-changed) Poisson processes on $\R_+$. Interestingly, this treatment is analogous to the statement of the Gaussian process limit in Section \ref{s:subcritical}, and those wishing a deeper exploration of this in a similar setting should refer to \cite{owaSPB}. What is precisely meant by this can be seen in the following proposition.

\begin{proposition}\label{P:poiproc}
The process $\mathcal{V}_k^{\pm}$ can be expressed as 
\[
(\mathcal{V}^{\pm}_k(t), t \geq 0) \overset{d}{=} \Big(N_k^{\pm}(t^{d(k+1)}), t \geq 0\Big),
\]
where $N_k^{\pm}$ is a (homogeneous) Poisson process with intensity $C_{f,k}m_k(D_1^{\pm})$ with $D_t^\pm = \{ \by \in \bbr^{d(k+1)}: h_t^\pm (0,\by) = 1 \}$. 
\end{proposition}
\begin{proof}
As with Proposition~\ref{P:gauproc}, we prove only the result for $\mathcal{V}^{+}_k$, as the proof for $\mathcal{V}^{-}_k$ is the same. We can see that if $0=t_0 < t_1 < \dots < t_k <\infty$ and $\lambda_i >0$, $i=1,\dots,k$, then by the non-decreasingness of $h_t^+$, 
\begin{align*}
\E\Bigl [ \exp\Bigl( -\sum_{i=1}^k \lambda_i \bigl(\mathcal{V}_k^+(t_i) - \mathcal{V}_k^+(t_{i-1})\bigl)\Bigr)\Big] = \E\Bigl [ \exp\Bigl( -\sum_{i=1}^k \lambda_i M_k (D_{t_i}^+ \setminus D_{t_{i-1}}^+) \Big)\Big],
\end{align*}
where $D_{t_i}^+ \setminus D_{t_{i-1}}^+$ are disjoint and $M_k (D_{t_i}^+ \setminus D_{t_{i-1}}^+)$, $i=1,\dots,k$, are independent. Moreover, $M_k (D_{t_i}^+ \setminus D_{t_{i-1}}^+)$ is Poisson distributed with parameter 
$$
C_{f,k} m_k (D_{t_i}^+ \setminus D_{t_{i-1}}^+) = C_{f,k} m_k (D_1^+) (t_i^{d(k+1)} - t_{i-1}^{d(k+1)})
$$
by a change of variable. Hence we have that
$$
\E\Bigl [ \exp\Bigl( -\sum_{i=1}^k \lambda_i M_k (D_{t_i}^+ \setminus D_{t_{i-1}}^+) \Big)\Big] = \prod_{i=1}^k \exp \Big( C_{f,k} m_k (D_1^+) (t_i^{d(k+1)} - t_{i-1}^{d(k+1)}) (e^{-\lambda_i} - 1)  \Big), 
$$
which implies that the process $\V_k^+(t^{1/d(k+1)})$ has independent increments and \[
\V_k^+((t+s)^{1/d(k+1)}) - \V_k^+(s^{1/d(k+1)})
\]
is Poisson with parameter $C_{f,k} m_k (D_1^+)t$. 
\end{proof}

In what follows we assume $\rho_n = 1$, 
though we could easily modify this to suppose that $\rho_n \to 1$ as $n \to \infty$. For simplicity in our proofs we assert the former. The proof is again given in Section \ref{s:proofs} and the main techniques there are those in \cite{dfsch16}. 
\begin{theorem}\label{T:Poisson}
Suppose that $\rho_n = 1$ and $f$ is an almost everywhere bounded and continuous density function. Then, we have the following weak convergence in the finite dimensional sense, namely
\[
\beta_{k,n}(t) \overset{fidi}{\Rightarrow} \mathcal{V}_k(t),
\]
meaning that for every $m \in \mathbb{N}$ and $0 < t_1 < t_2 < \dots < t_m < \infty$ we have 
\begin{equation}  \label{e:Poisson.beta}
\bigl( \beta_{k,n}(t_i), \, i=1,\dots,m \bigr) \Rightarrow \bigl(\mathcal{V}_k(t_i), \,  i=1,\dots,m\bigr),
\end{equation} 
weakly in $\R^m$. 
\end{theorem} 
\section{Proofs}  \label{s:proofs}

In this section we prove the theorems seen in the sections above, with the exposition focused on the critical and Poisson regimes. We only briefly discuss the sparse regime, since the proof is considerably similar to (or even easier than) the critical regime case. 

In the sequel,  we write $x + \by = (x + y_1, \dots, x+y_m)$ for $x\in \bbr^d$ and $\by = (y_1,\dots,y_m) \in \bbr^{dm}$. 

\subsection{Critical regime}  \label{s:proof.critical}
The first step towards the required central limit theorem is to examine the asymptotic moments as follows. 
Before proceeding with the proof, let us define the ``truncated'' Betti numbers 
\begin{equation}  \label{e:truncated.Betti}
\beta_{k,n,A}^{(M)}(t) := \sum_{i=k+2}^M \sum_{j>0} j U_{i,j,n,A}(t), \ \ M \in \bbn \cup \{ \infty \}
\end{equation}
for any measurable $A \subset \bbr^d$. 
Clearly $\beta_{k,n,A}(t) = \beta_{k,n,A}^{(\infty)}(t)$. 

Let us introduce a few items useful for specifying the limiting covariances. In the following $i, i_1, i_2, j_1$, and $j_2$ are positive integers, $t_1, t_2$ are non-negative reals, $A$ is an open subset of $\bbr^d$ with $m(\partial A) = 0$, and $a \wedge b := \min \{ a,b \}$ with $a\vee b := \max \{ a,b \}$. Additionally, we define the two functions
\begin{align}
\eta_{k,A}^{(i,j_1,j_2)} (t_1,t_2) &:= \int_{\bbr^{d(i-1)}} \int_{\bbr^d}  \one \bigl\{ \check{C} \big( \{ 0,\by \}, t_1 \wedge t_2 \big) \text{ is connected} \bigr\} \prod_{\ell=1}^2 b_{j_\ell, t_\ell}(0,\by) \label{e:eta.k}\\
&\qquad \times \exp \Big( -(t_1 \vee t_2)^d f(x) m \big( \B (\{ 0,\by \}; 1) \big) \Big)f(x)^i \one_A (x)\dif x \dif \by, \notag
\end{align}
and
\begin{align}
\nu_{k,A}^{(i_1,i_2,j_1,j_2)}(t_1,t_2) &:= \int_{\bbr^d}\dif x \int_{\bbr^{d(i_1-1)}} \dif \by_1 \int_{\bbr^{di_2}} \dif \by_2 \, \one \bigl\{ \check{C} \big( \{ 0,\by_1 \}, t_1 \big) \text{ is connected} \bigr\}\, \label{e:nu.k} \\
&\times \one \bigl\{ \check{C} \big(  \by_2 , t_2 \big) \text{ is connected} \bigr\} b_{j_1,t_1}(0,\by_1)\, b_{j_2, t_2}(\by_2) \notag \\
&\times \Big[ \Big( \alpha_{t_1,t_2} \big( \{ 0,\by_1 \}, \by_2 \big) - \alpha_{(t_1 \vee t_2) / 2} \big(  \{ 0,\by_1 \}, \by_2  \big) \Big)e^{-f(x) m \big( \B ( \{ 0,\by_1 \}; t_1 ) \cup \B(\by_2; t_2) \big)} \notag \\
&- \alpha_{t_1,t_2} \big( \{ 0,\by_1 \}, \by_2 \big) e^{-f(x) \big\{ m( \B(\{0,\by_1 \}; t_1) ) + m( \B(\by_2; t_2) ) \big\}}\Big] f(x)^{i_1 + i_2} \one_A(x), \notag 
\end{align}
where 
\begin{equation}  \label{e:union.balls}
\B(\mathcal X; r) := \bigcup_{y \in \mathcal X} B(y; r)
\end{equation} 
for a collection $\mathcal X$ of $\bbr^d$-valued vectors and $r>0$. Moreover, 
$$
\alpha_{r,s} (\Xio, \Xit):= \one \big\{ \B (\Xio; r)\cap \B(\Xit; s) \neq \emptyset \big\},
$$
and $\alpha_r(\Xio, \Xit) := \alpha_{r,r}(\Xio, \Xit)$. Finally we define for $M \in \bbn \cup \{ \infty \}$, 
$$
\Phi_{k,A}^{(M)} (t_1,t_2) := \sum_{i_1=k+2}^M \sum_{i_2=k+2}^M \sum_{j_1 > 0} \sum_{j_2 > 0} j_1 j_2\, \bigg( \frac{\etakrd^{(i_1, j_1, j_2)}(t_1, t_2) \delta_{i_1, i_2}}{i_1!}+ \frac{\nukrd^{(i_1, i_2, j_1, j_2)}(t_1, t_2)}{i_1! i_2!} \bigg)
$$ 
where $\delta_{i_1, i_2}$ is again the Kronecker delta and we define $\Phi_{k,A} (t_1,t_2) := \Phi_{k,A}^{(\infty)}(t_1,t_2)$. 
\begin{proposition}\label{P:covcritical}
Let f be an almost everywhere bounded and continuous density function. Let $ns_n^d = 1$ and $A \subset \mathbb{R}^d$ is open with $m(\partial A) = 0$. \\
$(i)$ If $M < \infty$, then for $t, t_1,t_2>0$, 
\[
n^{-1}\E[\beta_{k,n,A}^{(M)}(t)] \to \sum_{i=k+2}^M \sum_{j>0} \frac{j}{i!}\, \eta^{(i,j,j)}_{k,A}(t, t), \quad n \to \infty
\]
\[
n^{-1}\Cov(\beta_{k,n,A}^{(M)}(t_1), \beta_{k,n,A}^{(M)}(t_2)) \to \Phi_{k, A}^{(M)}(t_1, t_2), \quad n \to\infty. 
\]
$(ii)$ If $M = \infty$, then for $0 < t, t_1, t_2 < \big( e \| f \|_\infty \theta_d \big)^{-1/d}$, 
\[
n^{-1}\E[\beta_{k,n,A}(t)] \to \sum_{i=k+2}^{\infty} \sum_{j>0}  \frac{j}{i!}\, \eta^{(i,j,j)}_{k,A}(t, t), \quad n \to \infty
\]
\[
n^{-1}\Cov(\beta_{k,n,A}(t_1), \beta_{k,n,A}(t_2)) \to \Phi_{k, A}(t_1, t_2), \quad n \to\infty
\]
so that the limits above are finite non-zero constants. 
\end{proposition}

\begin{proof} 
We only establish the statements in $(ii)$. 
We aim to demonstrate the convergence of the expectation in Part 1 and then in Part 2, the convergence of the covariance to $\Phi_{k,A}(t_1, t_2)$. For ease of description we treat only the case when $A = \R^d$. The argument for a general $A$ will be the same except obvious minor changes. 
\vspace{5pt}

\noindent \emph{\underline{Part 1}}: The definition in \eqref{e:Betti1}, the Palm theory for Poisson processes in \cite{penr}, and the monotone convergence theorem supply that
\begin{equation}  \label{e:expectation.critical}
n^{-1}\E[\beta_{k,n}(t)] = \sum_{i=k+2}^{\infty} \sum_{j>0} j\, \frac{n^{i-1}}{i!}\, \E[\grntij(\X i, \X i \cup \Pn)], 
\end{equation}
where $\X i = (X_1,\dots,X_i) \in \bbr^{di}$ is a collection of i.i.d random points in $\bbr^d$ with common density $f$. 
By conditioning on $\X i$ we have that
\begin{align}
&n^{i-1}\E[\grntij(\X i, \X i \cup \Pn)]   \label{e:critical.exp} \\
&= n^{i-1} \E \Big[  b_{j,r_n(t)}(\X i) \E \big[ J_{i,r_n(t)}(\X i, \X i \cup \Pn)\, \big|\, \X i  \big] \Big] \notag \\
&= n^{i-1} \int_{\bbr^{di}} \one \big\{ \check{C} (\bx, r_n(t)) \text{ is connected}  \big\}\, b_{j,r_n(t)} (\bx) \exp \big(  -nI_{r_n(t)} (\bx)\big) \prod_{j=1}^i f(x_j) \dif \bx, \notag
\end{align}
where 
$$
I_{r_n(t)}(\bx)  =  I_{r_n(t)}(x_1,\dots,x_i) = \int_{\B (\bx; r_n(t))} f(z) \dif{z}.
$$
Subsequently we perform the change of variables $x_1 = x$ and $x_j = x + s_ny_{j-1}$ for $j=2,\dots, i$, to get that \eqref{e:critical.exp} is equal to 
\begin{align*}
&(ns_n^d)^{i-1} \int_{\bbr^{d(i-1)}}\int_{\bbr^{d}} \one \big\{  \check{C} (\{ x,x+s_n\by \}, r_n(t)) \text{ is connected}\big\} b_{j,r_n(t)}(x,x+s_n \by) \\
&\qquad \qquad \qquad \qquad \times \exp \big(  -nI_{r_n(t)}(x,x+s_n\by) \big) f(x) \prod_{j=1}^{i-1} f(x+s_n y_j) \dif x \dif \by \\
&=  \int_{\bbr^{d(i-1)}}\int_{\bbr^{d}} \one \big\{  \check{C} (\{ 0,\by \}, t) \text{ is connected}\big\} b_{j,t}(0,\by) \\
&\qquad \qquad \qquad \qquad \times \exp \big(  -nI_{r_n(t)}(x,x+s_n\by) \big) f(x) \prod_{j=1}^{i-1} f(x+s_n y_j) \dif x \dif \by, 
\end{align*}
where the equality follows from the location and scale invariance of both of the indicator functions. By the continuity of $f$ we have that $\prod_{j=1}^{i-1} f(x+s_n y_j) \to f(x)^{i-1}$ a.e. as $n\to\infty$. 
As for the convergence of the exponential term, we have 
\[
nI_{r_n(t)}(x, x + s_n\by) = n\int_{\B (\{ x, x+s_n \by \}; r_n(t))} f(z) \dif z,
\]
which after the change of variable $z = x + s_n v$, gives us 
\[
n\int_{\B (\{ x, x+s_n \by \}; r_n(t))} f(z) \dif z \to t^d f(x) m \Big( \B \big( \{ 0,\by \}; 1 \big) \Big). 
\]
It then follows from the dominated convergence theorem that 
$$
n^{i-1}\E[\grntij(\X i, \X i \cup \Pn)] \to \etakrd^{(i,j,j)}(t,t), \ \ n\to\infty. 
$$
It remains to find a summable upper bound for \eqref{e:expectation.critical} to apply the dominated convergence theorem for sums. To this end we use the inequality $j \leq \binom{i}{k+1}$ which is the result of the fact that there must be a $k$-simplex in $\check{C} (\mathcal X_i, r_n(t))$ whenever $\beta_k \bigl( \check{C} (\mathcal X_i, r_n(t)) \bigr)>0$. In addition, using an obvious inequality 
\begin{equation} \label{e:obvbound}
J_{i,r_n(t)}(\mathcal X_i, \mathcal X_i \cup \Pn) \leq \one \bigl\{  \C (\mathcal X_i, r_n(t)) \text{ is connected} \bigr\}, 
\end{equation}
we get that 
\begin{align}
n^{-1} \E [\beta_{k,n}(t)] &\leq \sum_{i=k+2}^\infty \binom{i}{k+1}  \frac{n^{i-1}}{i!} \sum_{j=1}^{\binom{i}{k+1}}  \E \bigl[ \, \one  \bigl\{  \C (\mathcal X_i, r_n(t)) \text{ is connected} \bigr\}\, b_{j,r_n(t)} (\mathcal X_i) \bigr] \label{e:bound.A1n}\\
&\leq \sum_{i=k+2}^\infty \binom{i}{k+1}  \frac{n^{i-1}}{i!}\, \P \bigl( \C (\mathcal X_i, r_n(t)) \text{ is connected}  \bigr). \notag
\end{align}
For further analysis we claim that 
\begin{equation}  \label{e:spanning.tree}
\P \bigl( \C (\mathcal X_i, r_n(t)) \text{ is connected}  \bigr) \leq i^{i-2} \bigl( r_n(t)^d \| f \|_\infty \theta_d  \bigr)^{i-1}. 
\end{equation}
Indeed this can be derived from 
\begin{align}
&\P \bigl( \C (\mathcal X_i, r_n(t)) \text{ is connected}  \bigr) \label{e:connectedness} \\
&= \int_{\bbr^{di}} \one \bigl\{  \C (\bx, r_n(t)) \text{ is connected} \bigr\}\, \prod_{j=1}^i f(x_j) \dif \bx \notag \\
&= r_n(t)^{d(i-1)} \int_{\bbr^{di}} \one \bigl\{  \C (\{ 0,\by \}, 1) \text{ is connected} \bigr\}\, f(x) \prod_{j=1}^{i-1} f(x+r_n(t) y_j) \dif x \dif \by \notag \\
&\leq \bigl(r_n(t)^d \| f \|_\infty \bigr)^{i-1} \int_{\bbr^{d(i-1)}} \one \bigl\{  \C (\{ 0,\by \}, 1) \text{ is connected} \bigr\} \dif \by \notag \\
&\leq i^{i-2} \bigl( r_n(t)^d \| f \|_\infty \theta_d \bigr)^{i-1}. \notag
\end{align}
The last inequality comes from the basic fact that there are $i^{i-2}$ spanning trees on $i$ vertices. 
Combining \eqref{e:bound.A1n}, \eqref{e:spanning.tree}, and $ns_n^d=1$ we conclude that 
\begin{align}
n^{-1} \E [\beta_{k,n}(t)] &\leq \frac{1}{(k+1)!}\sum_{i=k+2}^{\infty} \frac{i^{i-2}}{(i-k-1)!} (t^d \| f \|_{\infty} \theta_d)^{i-1} =: \frac{1}{(k+1)!}\sum_{i=k+2}^{\infty}a_i. \notag
\end{align}
It is easy to check that  $a_{i+1}/a_i \to et^d \| f \|_\infty \theta_d$ as $i\to\infty$, where the limit is less than $1$ by our assumption. So the ratio test has shown that $\sum_{i=k+2}^\infty a_i$ converges as required.  
\vspace{5pt}

\noindent \emph{\underline{Part 2}}: We assume $0 < t_1 \leq t_2 < (e \lVert f \rVert_{\infty} \theta_d)^{-1/d}$ and proceed with the fact that 
\begin{align*}
&\E[\beta_{k,n}(t_1)\beta_{k,n}(t_2)] \label{P:covsparse_skn} \\
&=  \sum_{i_1=k+2}^\infty  \sum_{i_2=k+2}^\infty \sum_{j_1>0} \sum_{j_2>0} j_1 j_2 \E\left[ \sum_{\Y_1 \subset \mathcal{P}_n} \sum_{\Y_2 \subset \mathcal{P}_n} \gone(\Y_1, \mathcal{P}_n)\, \gtwo(\Y_2, \mathcal{P}_n)  \right ] \notag \\
& = \sum_{i=k+2}^\infty \sum_{j_1>0} \sum_{j_2>0} j_1 j_2 \E\Bigg[ \sum_{\Y \subset \Pn} g_{r_n(t_1)}^{(i,j_1)}(\Y, \mathcal{P}_n)\, g_{r_n(t_2)}^{(i,j_2)}(\Y, \mathcal{P}_n)  \notag \Bigg]\\
&+ \sum_{i_1=k+2}^\infty  \sum_{i_2=k+2}^\infty \sum_{j_1>0} \sum_{j_2>0} j_1 j_2 \E \Biggl[ \sum_{\Y_1 \subset \Pn} \sum_{\Y_2 \subset \mathcal{P}_n} \gone(\Y_1, \mathcal{P}_n)\, \gtwo(\Y_2, \mathcal{P}_n) \one \bigl\{ |\Y_1 \cap \Y_2| = 0\bigr\} \Bigg].  \notag
\end{align*}
The second equality comes from an observation that if $\Y_1 \neq \Y_2$ and the intersection of $\Y_1$ and $\Y_2$ is non-empty, then $\check{C} ( \Y_2, r_n(t_2))$ cannot be an isolated component of $\check{C}(\Pn, r_n(t_2))$---so these terms are zero. Appealing to Palm theory again as seen in \cite{km}, we get that 
\begin{align*}
&\E[\beta_{k,n}(t_1)\beta_{k,n}(t_2)] \\
& = \sum_{i=k+2}^\infty \sum_{j_1>0} \sum_{j_2>0} j_1 j_2 \frac{n^i}{i!}\, \E\Big[  g_{r_n(t_1)}^{(i,j_1)}(\X i, \X i \cup \mathcal{P}_n)\,  g_{r_n(t_2)}^{(i,j_2)}(\X i, \X i \cup \mathcal{P}_n)  \notag \Big]\\
&+ \sum_{i_1=k+2}^\infty  \sum_{i_2=k+2}^\infty \sum_{j_1>0} \sum_{j_2>0} j_1 j_2\, \frac{n^{i_1 + i_2}}{i_1! i_2!}\, \\
&\qquad \qquad \qquad \qquad \times  \E \Bigl[  \gone(\Xio, \Xiot \cup \mathcal{P}_n)\, \gtwo(\Xit, \Xiot \cup \mathcal{P}_n) \Big], \notag
\end{align*}
where $\X i$ and $\Pn$ are independent, and $\Xio$, $\Xit$, and $\Pn$ are also mutually independent such that $\Xio$ and $\Xit$ are disjoint.  

Applying \eqref{e:expectation.critical} to each $\E[\beta_{k,n}(t_i)]$, $i=1,2$, and utilizing the independence of $\Xio$ and $\Xit$, we see that the covariance function can be written as 
\begin{equation}  \label{e:covariance.into.two}
\Cov(\beta_{k, n}(t_1), \beta_{k,n}(t_2)) = A_{1,n} + A_{2,n}, 
\end{equation}  
with
\begin{align}
A_{1,n} &:= \sum_{i=k+2}^\infty \sum_{j_1>0} \sum_{j_2>0} j_1j_2\frac{n^i}{i!}\, \E \bigl[ g_{r_n(t_1)}^{(i, j_1)} (\mathcal X_i, \mathcal X_i \cup \Pn) g_{r_n(t_2)}^{(i, j_2)} (\mathcal X_i, \mathcal X_i \cup \Pn)  \bigr], \label{e:A1n} \\
A_{2,n} &:= \sum_{i_1 = k+2}^{\infty} \sum_{i_2 = k+2}^{\infty} \sum_{j_1 >0} \sum_{j_2 >0} j_1 j_2\,  \frac{n^{i_1 + i_2}}{i_1! i_2!}\, \label{e:A2n}\\
&\phantom{{{{{{A_{2,n} := \sum_{i_1 = k+2}^{\infty} \sum_{i_2 = k+2}^{\infty}}}}}}} \times \E\big[\gone(\Xio, \Xiot \cup \Pn)\gtwo(\Xit, \Xiot \cup \Pn) \notag \\
&\phantom{{{{{{A_{2,n} := \sum_{i_1 = k+2}^{\infty} \sum_{i_2 = k+2}^{\infty} \times E[\gone(}}}}}} - \gone(\Xio, \Xio \cup \Pn)\gtwo(\Xit, \Xit \cup \Pn')\big],\notag 
\end{align}
where $\Pn'$ is an independent copy of $\Pn$ and is also independent of $\Xio$ and $\Xit$. 

Let us denote the expectation portions of $A_{1,n}$ and $A_{2,n}$ as $\Eonij$ and $\Etnij$, with $\bi = (i_1,i_2)$, and $\bj = (j_1,j_2)$ respectively. Our goal is to show that $n^{-1} (A_{1,n} + A_{2,n})$ tends to $\Phi_{k,\bbr^d}(t_1,t_2)$ as $n\to\infty$. For now we shall compute the limits of $n^{i-1} \Eonij$ and $n^{i_1+i_2-1} \Etnij$ for each $i, i_1, i_2, j_1$, and $j_2$, while temporarily assuming that the dominated convergence theorem for sums is applicable for both $n^{-1}A_{1,n}$ and $n^{-1}A_{2,n}$.  
By mirroring the argument from Part 1 with the same change of variables and recalling $t_1 \leq t_2$, 
\begin{align*}
n^{i-1} \Eonij &= n^{i-1} \E \Big[ \one \big\{ \check{C} (\X i, r_n(t_1)) \text{ is connected}  \big\}  \prod_{\ell=1}^2 b_{j_\ell, r_n(t_\ell)} (\X i) \exp \big(  -nI_{r_n(t_2)} (\X i) \big)\Big] \\
&=\int_{\bbr^{d(i-1)}}\int_{\bbr^d} \one \big\{ \check{C} \big( \{ 0,\by \}, t_1 \big) \text{ is connected} \big\}\, \prod_{\ell=1}^2 b_{j_\ell, t_\ell} (0,\by) \\
&\qquad \qquad  \quad \times \exp \big( -nI_{r_n(t_2)}(x,x + s_n \by) \big)f(x) \prod_{j=1}^{i-1} f(x+s_n y_j) \dif x \dif \by \\
&\to  \etakrd^{(i,j_1,j_2)}(t_1, t_2)  \ \ \text{as } n\to\infty.
\end{align*}
Hence the assumed dominated convergence theorem for sums concludes that
\begin{equation}  \label{e:conv.A1n}
n^{-1}A_{1,n} \to \sum_{i=k+2}^\infty \sum_{j_1 >0}\sum_{j_2 >0} \frac{j_1j_2}{i!}\, \etakrd^{(i, j_1, j_2)}(t_1, t_2)\, \quad n \to \infty.
\end{equation}

To demonstrate convergence for $n^{i_1+i_2-1} \Etnij$, let us shorten $\gone$ to $g_1$ and $\gtwo$ to $g_2$ and decompose $\Etnij$ into two terms:
\begin{align*}
\Etnij &= \E \bigl[ g_1(\Xio, \Xiot \cup \Pn) g_2(\Xit, \Xiot \cup \Pn)   -  g_1(\Xio, \Xio \cup \Pn) g_2(\Xit, \Xit \cup \Pn)  \bigr] \notag  \\
&+ \E\bigl[g_1(\Xio, \Xio \cup \mathcal{P}_n)\bigl( g_2(\Xit, \Xit \cup \mathcal{P}_n) - g_2(\Xit, \Xit \cup \mathcal{P}^{\prime}_n)\bigr)\bigr]  \notag  \\
&:= B_{1,n} + B_{2,n}. \notag 
\end{align*}
Note that for $\ell=1,2,$
$$
g_\ell (\mathcal X_{i_\ell}, \Xiot \cup \Pn) = g_\ell (\mathcal X_{i_\ell}, \mathcal X_{i_\ell} \cup \Pn)\, \one \bigl\{ \B\bigl(\Xio; r_n(t_\ell)/2\bigr)\cap \B\bigl(\Xit; r_n(t_\ell)/2\bigr) = \emptyset \bigr\}, 
$$
where $\B (\mathcal X; r)$ is defined  in \eqref{e:union.balls}. 
Hence we have that 
\begin{equation*}  
B_{1,n} = - \E \bigl[ g_1(\Xio,\Xio \cup \Pn)\, g_2(\Xit, \Xit \cup \Pn)\, \alpha_{r_n(t_2)/2} (\Xio, \Xit) \bigr]. 
\end{equation*}
At the same time, the spatial independence of $\Pn$ justifies  that
\begin{align*}
B_{2,n} &= \E\bigl[g_1(\Xio, \Xio \cup \mathcal{P}_n)\bigl(g_2(\Xit, \Xit \cup \mathcal{P}_n) - g_2(\Xit, \Xit \cup \mathcal{P}^{\prime}_n)\bigr)\, \alpha_{r_n(t_1), r_n(t_2)} (\Xio, \Xit)\big]. 
\end{align*}
Consequently we can rewrite $\Etnij$ as 
\begin{align}
\Etnij&= \E \Big[  g_1 (\Xio, \Xio \cup \Pn) g_2 (\Xit, \Xit \cup \Pn)   \big( \alpha_{r_n(t_1), r_n(t_2)}(\Xio, \Xit) - \alpha_{r_n(t_2)/2}(\Xio, \Xit)  \big) \Big]   \label{e:Etnij} \\
&\qquad - \E \big[ g_1 (\Xio, \Xio \cup \Pn) g_2 (\Xit, \Xit \cup \Pn') \alpha_{r_n(t_1), r_n(t_2)} (\Xio, \Xit)  \big] \notag \\
&:= C_{1,n} - C_{2,n}. \notag
\end{align}
After conditioning on $\Xiot$, the customary change of variable yields
\begin{align*}
n^{i_1 + i_2 - 1}C_{1,n} &= n^{i_1 + i_2 - 1} \E \bigg[ \prod_{\ell=1}^2 \one \big\{ \check{C} (\mathcal X_{i_\ell}, r_n(t_\ell)) \text{ is connected} \big\}\, b_{j_\ell, r_n(t_\ell)} (\mathcal X_{i_\ell})  \\
&\qquad \qquad \qquad  \times \bigl( \alpha_{r_n(t_1), r_n(t_2)}(\Xio, \Xit) - \alpha_{r_n(t_2)/2} (\Xio, \Xit)  \big) \\
&\qquad \qquad \qquad  \times \exp \Big( -n\int_{\B (\Xio; r_n(t_1)) \cup \B(\Xit; r_n(t_2))} f(z)\dif z \Big)\bigg] \\
&= \int_{\bbr^d} \dif x \int_{\bbr^{d(i_1 -1)}} \dif \by_1 \int_{\bbr^{di_2}} \dif \by_2 \, \one \bigl\{ \check{C} \big( \{ 0,\by_1 \}, t_1 \big) \text{ is connected} \bigr\}\, \\
&\qquad \times \one \bigl\{ \check{C} \big(  \by_2, t_2 \big) \text{ is connected} \bigr\} b_{j_1,t_1}(0,\by_1)\, b_{j_2, t_2}(\by_2) \\
&\qquad \times \Big( \alpha_{t_1,t_2} \big(\{ 0,\by_1 \}, \by_2\big)  - \alpha_{t_2/2} \big(\{0,\by_1 \}, \by_2\big) \Big) \\
&\qquad \times \exp \Big( -n\int_{\B (\{x,x+s_n \by_1  \}; r_n(t_1)) \cup \B (x + s_n \by_2; r_n(t_2)) }f(z)\dif z \Big) \\
&\qquad \times  f(x) \prod_{j=1}^{i_1-1} f(x + s_n y_{1,j})\prod_{j=1}^{i_2} f(x + s_n y_{2,j}) \\
&\to\int_{\bbr^d} \dif x \int_{\bbr^{d(i_1 -1)}} \dif \by_1 \int_{\bbr^{di_2}} \dif \by_2 \, \one \bigl\{ \check{C} \big( \{ 0,\by_1 \}, t_1 \big) \text{ is connected} \bigr\}\, \\
&\qquad \times \one \bigl\{ \check{C} \big(  \by_2, t_2 \big) \text{ is connected} \bigr\} b_{j_1,t_1}(0,\by_1)\, b_{j_2, t_2}(\by_2) \\
&\qquad \times \Big( \alpha_{t_1,t_2} \big(\{ 0,\by_1 \}, \by_2\big)  - \alpha_{t_2/2} \big(\{0,\by_1 \}, \by_2\big) \Big) \\
&\qquad \times e^{-f(x) m \big( \B (\{ 0,\by_1 \}; t_1) \cup \B (\by_2; t_2) \big)} f(x)^{i_1 + i_2}, 
\end{align*}
where $\by_1 = (y_{1,1}, \dots, y_{1,i_1-1}) \in \bbr^{d(i_1-1)}$ and $\by_2 = (y_{2,1}, \dots, y_{2,i_2}) \in \bbr^{di_2}$. \\
Similarly  one can see that
\begin{align*}
n^{i_1 + i_2 - 1}C_{2,n} &\to \int_{\bbr^d} \dif x \int_{\bbr^{d(i_1 -1)}} \dif \by_1 \int_{\bbr^{di_2}} \dif \by_2 \, \one \bigl\{ \check{C} \big( \{ 0,\by_1 \}, t_1 \big) \text{ is connected} \bigr\}\, \\
&\qquad \times \one \bigl\{ \check{C} \big(  \by_2, t_2 \big) \text{ is connected} \bigr\} b_{j_1,t_1}(0,\by_1)\, b_{j_2, t_2}(\by_2) \alpha_{t_1, t_2} \big( \{ 0,\by_1 \}, \by_2 \big)\\
&\qquad \times e^{-f(x) \big\{ m(\B(\{ 0,\by_1 \}; t_1)) + m (\B(\by_2; t_2)) \big\}} f(x)^{i_1 + i_2}. 
\end{align*}
Therefore, 
$$
n^{i_1 + i_2 - 1} \Etnij = n^{i_1 + i_2 - 1}( C_{1,n} - C_{2,n} ) \to \nukrd^{(i_1,i_2,j_1,j_2)} (t_1,t_2), \ \ \ n\to\infty.  
$$
Assuming convergence under summation, we have that 
\begin{equation}  \label{e:conv.A2n}
n^{-1}A_{2,n} \to \sum_{i_1=k+2}^\infty \sum_{i_2 = k+2}^\infty \sum_{j_1>0} \sum_{j_2 > 0} \frac{j_1 j_2}{i_1! i_2!}\, \nukrd^{(i_1, i_2, j_1, j_2)}(t_1, t_2), \ \ n\to\infty. 
\end{equation}
From \eqref{e:conv.A1n} and \eqref{e:conv.A2n}, it follows that $n^{-1} (A_{1,n} + A_{2,n}) \to \Phi_{k,\bbr^d}(t_1, t_2)$ as $n\to\infty$. 

Now we would like to show that both $n^{i-1} \Eonij$ and $n^{i_1 + i_2-1}| \Etnij|$ are bounded by a summable quantity, so that application of the dominated convergence theorem for sums is valid for both $n^{-1} A_{1,n}$ and $n^{-1} A_{2,n}$. 
Using the bounds \eqref{e:obvbound}, \eqref{e:spanning.tree}, together with $ns_n^d=1$, we have
\begin{align}
n^{-1} A_{1,n} &\leq \sum_{i=k+2}^\infty \sum_{j_1>0} \sum_{j_2>0}j_1 j_2\,  \frac{n^{i-1}}{i!}\, \E\bigg[ \one \bigl\{  \C (\mathcal{X}_i, r_n(t_1)) \text{ is connected} \bigr\} \prod_{\ell=1}^{2} b_{j_{\ell}, r_n(t_{\ell})}(\mathcal{X}_i) \bigg] \label{e:A1n.bound}\\
&\leq \sum_{i = k+2}^{\infty} \binom{i}{k+1}^2 \frac{n^{i-1}}{i!}\, \P\big(  \C (\mathcal X_{i}, r_n(t_1)) \text{ is connected} \big) \notag \\
&\leq \frac{1}{\big((k+1)!\big)^2} \sum_{i=k+2}^\infty \frac{i! i^{i-2}}{\big( (i-k-1)! \big)^2}\, \big( t_1^d \| f \|_\infty \theta_d \big)^{i-1}. \notag
\end{align}
The last term is convergent by appealing to the assumption $t_1 < (e \| f \|_\infty \theta_d)^{-1/d}$ and the ratio test for sums. 

Subsequently we turn our attention to $n^{-1}A_{2,n}$. Returning to \eqref{e:Etnij} and using obvious relations 
$$
\alpha_{r_n(t_1), r_n(t_2)} (\Xio, \Xit) \leq \alpha_{r_n(t_2)} (\Xio, \Xit), \quad \alpha_{r_n(t_2)/2} (\Xio, \Xit) \leq \alpha_{r_n(t_2)} (\Xio, \Xit),
$$
we get that
$$
|C_{1,n} - C_{2,n}| \leq 3 \E \Big[ \prod_{\ell=1}^2 \one \big\{  \check{C} (\mathcal X_{i_\ell}, r_n(t_2)) \text{ is connected} \big\}\, b_{j_\ell, r_n(t_\ell)} (\mathcal X_{i_\ell})\, \alpha_{r_n(t_2)}(\Xio, \Xit) \Big]
$$
By virtue of this bound we have that
\begin{align}
n^{-1} |A_{2,n}| &\leq 3 \sum_{i_1 = k+2}^{\infty} \sum_{i_2 = k+2}^{\infty} \sum_{j_1>0} \sum_{j_2>0} j_1 j_2\,  \frac{n^{i_1 + i_2 - 1}}{i_1! i_2!} \label{e:A2n.bound} \\
&\times \E \Big[  \prod_{\ell=1}^2 \one \big\{  \check{C} (\mathcal X_{i_\ell}, r_n(t_2)) \text{ is connected} \big\}\, b_{j_\ell, r_n(t_\ell)} (\mathcal X_{i_\ell})\, \alpha_{r_n(t_2)}(\Xio, \Xit) \Big] \notag \\
&\leq 3 \sum_{i_1 = k+2}^{\infty} \sum_{i_2 = k+2}^{\infty} \binom{i_1}{k+1} \binom{i_2}{k+1} \,  \frac{n^{i_1 + i_2 - 1}}{i_1! i_2!} \notag \\
&\times \P \Bigl( \C (\mathcal X_{i_\ell}, r_n(t_2)) \text{ is connected for } \ell=1,2, \ \B \bigl(\Xio; r_n(t_2)\bigr) \cap \B \bigl(\Xit; r_n(t_2)\bigr) \neq \emptyset \Bigr).  \notag
\end{align} 
We claim here that 
\begin{align}
&\P \Bigl( \C (\mathcal X_{i_\ell}, r_n(t_2)) \text{ is connected for } \ell=1,2, \ \B \bigl(\Xio; r_n(t_2)\bigr) \cap \B \bigl(\Xit; r_n(t_2)\bigr) \neq \emptyset \Bigr) \label{e:spanning.tree.bound2}\\
&\leq 2^d i_1^{i_1-1} i_2^{i_2-1} \bigl( r_n(t_2)^d \| f \|_\infty \theta_d \bigr)^{i_1+i_2-1}. \notag
\end{align}
To see this, by the change of variables as in \eqref{e:connectedness}, we have that 
\begin{align*} 
&\P \Bigl( \C (\mathcal X_{i_\ell}, r_n(t_2)) \text{ is connected for } \ell=1,2, \ \B \bigl(\Xio; r_n(t_2)\bigr) \cap \B \bigl(\Xit; r_n(t_2)\bigr) \neq \emptyset \Bigr) \notag \\
&\leq \bigl( r_n(t_2)^d \| f \|_\infty \bigr)^{i_1+i_2-1} \int_{\bbr^{d(i_1+i_2-1)}} \one \bigl\{ \C (\{0,y_1,\dots,y_{i_1-1} \}, 1) \text{ is connected} \bigr\} \notag \\
&\qquad \qquad \qquad \qquad \quad  \times \one \bigl\{ \C (\{y_{i_1},\dots,y_{i_1+i_2-1} \}, 1) \text{ is connected} \bigr\} \notag \\
&\qquad \qquad \qquad \qquad \quad  \times \one \bigl\{ \B (\{ 0,y_1,\dots,y_{i_1-1} \}; 1) \cap \B (\{ y_{i_1},\dots,y_{i_1+i_2-1} \}; 1) \neq \emptyset \bigr\} \dif \by. 
\end{align*}
Note that there are $i_1^{i_1-2}$ spanning trees on the set of points $\{ 0,y_1,\dots,y_{i_1-1} \}$ with unit connectivity radius, and there are $i_2^{i_2-2}$ spanning trees on $\{ y_{i_1}, \dots, y_{i_1 + i_2 -1} \}$ with unit connectivity radius as well. In addition there are $i_1 \times i_2$ possible ways of picking one vertex from $\{ 0,y_1,\dots,y_{i_1-1} \}$ and another from $\{ y_{i_1}, \dots, y_{i_1+i_2-1} \}$, and connecting the two chosen vertices with connectivity radius $2$. Therefore, the expression above is eventually bounded by 
$$
\bigl( r_n(t_2)^d \| f \|_\infty \bigr)^{i_1+i_2-1} i_1^{i_1-2} i_2^{i_2-2} \theta_d^{i_1 + i_2 - 2} (i_1 i_2 2^d \theta_d) = 2^d i_1^{i_1-1} i_2^{i_2-1} \bigl( r_n(t_2)^d \| f \|_\infty \theta_d \bigr)^{i_1+i_2-1}.
$$ 
Now we have 
$$
n^{-1} |A_{2,n}| \leq \frac{3\cdot 2^d}{\big( (k+1)! \big)^2 t_2^d \|f  \|_\infty \theta_d}\, \bigg\{ \sum_{i=k+2}^\infty \frac{i^{i-1}}{(i-k-1)!}\, \big( t_2^d \| f \|_\infty \theta_d \big)^i \bigg\}^2. 
$$
The constraint $t_2 < (e\| f \|_\infty \theta_d)^{-1/d}$, together with the ratio test, guarantees that the last term converges. Hence the proof is completed. 
\end{proof}

\begin{proof}[Proof of Theorem \ref{T:cclt}]
We begin by proving the corresponding result for the truncated Betti number in \eqref{e:truncated.Betti} for every $M \in \bbn$, that is, 
$$
n^{-1/2} \bigl( \beta_{k,n}^{(M)}(t_i)  -\E [\beta_{k,n}^{(M)}(t_i)], \, i=1,\dots,m \bigr) \Rightarrow \bigl( \Hk^{(M)}(t_i) \, i=1,\dots,m \bigr), 
$$
where $\Hk^{(M)}$ is the ``truncated" limiting centered Gaussian process given by
$$
\Hk^{(M)}(t) = \sum_{i=k+2}^M \sum_{j>0} j \Hk^{(i,j)}(t). 
$$ 

We now restrict ourselves to the case in which the corresponding left most points belong to a fixed bounded set $A$. By the Cram\'er-Wold device, we need to 
demonstrate a univariate central limit theorem for $\sum_{i=1}^m a_i \beta_{k,n,A}^{(M)}(t_i)$, where $a_i \in \bbr$, $m \geq 1$. The asymptotic variance of $\sum_{i=1}^m a_i \beta_{k,n,A}^{(M)}(t_i)$ scaled by $n^{-1/2}$ can be derived from Proposition \ref{P:covcritical}~$(i)$: 
\begin{align}
\begin{split}
\Var\Bigl(n^{-1/2} \sum_{i=1}^m a_i \beta_{k,n,A}^{(M)}(t_i) \Bigr) &= \sum_{i=1}^m \sum_{j=1}^m a_i a_j n^{-1} \Cov(\beta_{k,n,A}^{(M)}(t_i), \beta_{k,n,A}^{(M)}(t_j)) \\
&\to \sum_{i=1}^m \sum_{j=1}^m a_i a_j \Phi_{k,A}^{(M)}(t_i,t_j), \ \ \ n\to\infty. \label{T:sparse_gamma}
\end{split}
\end{align}

Our proof exploits Stein's normal approximation method for weakly dependent random variables, as in Theorem 2.4 in \cite{penr}. We assume the limit in \eqref{T:sparse_gamma} is positive as otherwise our proof is trivial. Define $t := \max\{t_1, \dots, t_m \} = t_m$ and let $(Q_{j,n}, \, j \in \mathbb{N})$ be an enumeration of almost disjoint closed cubes of side length $r_n(t)$, such that $\cup_{j \in \bbn} Q_{j,n} = \R^d$. Recalling $A$ is bounded, we define
\[
V_n := \{j \in \mathbb{N}: Q_{j,n} \cap A \neq \emptyset\},
\]
and 
\[
\xi_{j, n} :=  \sum_{i=1}^m  a_i \beta_{k,n,A\cap Q_{j,n}}^{(M)}(t_i),
\]
so that $\sum_{i=1}^m a_i \beta_{k,n,A}^{(M)}(t_i) = \sum_{j \in V_n} \xi_{j, n} $. We now turn $V_n$ into the vertex set of a \textit{dependency graph} (see Section 2.1 in \cite{penr} for the definition) by declaring that for $j, j' \in V_n$, $j \sim {\pr j}$ if and only if $d(Q_{j,n}, Q_{{\pr j}, n}) \leq 2Mr_n(t)$. It is easy to show that this provides us with the required independence properties, that is, for any vertex set $I_1, I_2 \subset V_n$ with no edges connecting vertices in $I_1$ and those in $I_2$, we have that $(\xi_{j,n}, \,  j \in I_1)$ and $(\xi_{j,n}, \,  j \in I_2)$ are independent. Note moreover that the degree of $(V_n, \sim)$ is uniformly bounded regardless of $n$. 
Since $A$ is a bounded set, we have $| V_n| = \mathcal O (s_n^{-d})$. Let $Y_{j,n}$ denote the number of points of $\Pn$ belonging to 
$$
\text{Tube}(Q_{j,n}, Mr_n(t)) := \bigl\{ x \in \bbr^d: \inf_{y \in Q_{j,n}} |x-y| \leq Mr_n(t)  \bigr\}. 
$$
Then we have
\begin{align*}
|\xi_{j,n}| &\leq \sum_{i=1}^m |a_i| \beta_{k,n, A\cap Q_{j,n}}^{(M)} (t_i) \\
&\leq \sum_{i=1}^m |a_i| \beta_k \Big(  \check{C} \big( \Pn \cap \text{Tube} \big( Q_{j,n}, Mr_n(t) \big), r_n(t_i) \big) \Big) \\
&\leq \sum_{i=1}^m |a_i| \binom{Y_{j,n}}{k+1}. 
\end{align*}
By definition, $Y_{j,n}$ is Poisson distributed with parameter
\[
\lambda_{j,n} := n\int_{\text{Tube}(Q_{j,n}, Mr_n(t))} f(z) \, \dif{z}, 
\]
which itself yields an upper bound of the form
\begin{equation} \label{e:poi.param.bound}
\lambda_{j,n} \leq n \| f \|_{\infty} m \Bigl(  \text{Tube}\bigl(Q_{j,n}, Mr_n(t)\bigr)\Bigr) := c. 
\end{equation}
This implies that $Y_{j,n}$ is stochastically dominated by a Poisson random variable, which we call $Y$, with parameter $c$. The assumption $ns_n^d=1$ ensures that $c$ does not depend on $n$, and for the rest of the proof, let $C^*$ denote a generic positive constant which is independent of $n$ but may vary between lines. 

We get that for $\alpha \in \mathbb{N}$
\begin{equation}  \label{e:bound.xijn}
\E [|\xi_{j,n}|^\alpha] \leq \Bigl( \sum_{i=1}^m |a_i| \Bigr)^\alpha \E \left[ \binom{Y_{j,n}}{k+1}^\alpha \right] 
\leq \Bigl( \sum_{i=1}^m |a_i| \Bigr)^\alpha \E \left[ \binom{Y}{k+1}^\alpha \right] =C^*
\end{equation} 
Letting 
\[
\pr{\xi}_{j,n} := \frac{\xi_{j,n} - \E[\xi_{j,n}]}{\sqrt{\Var(\sum_{i=1}^m a_i \beta_{k,n,A}^{(M)}(t_i))}},
\]
it is clear that $(V_n, \sim)$ still constitutes a dependency graph for the $(\pr{\xi}_{j,n}, \, j \in \mathbb{N})$ because independence is not affected by affine transformations. 
Let $Z$ be a standard normal random variable. 
It then follows from Stein's normal approximation method (i.e. Theorem 2.4 from \cite{penr}) that for all $x\in \bbr$, 
\begin{align*}
\Bigl|\,  \P\bigl( \sum_{j \in V_n} \xi_{j,n}' \leq x \bigr) - &\P(Z \leq x) \, \Bigr| \leq C^*\left( \sqrt{s_n^{-d} \E \bigl[ |\xi_{j,n}'|^3 \bigr]} + \sqrt{s_n^{-d} \E \bigl[ |\xi_{j,n}'|^4 \bigr]}   \right) \\
&\leq C^*\left( \sqrt{s_n^{-d}n^{-3/2} \E \bigl[ |\xi_{j,n}-\E[\xi_{j,n}]|^3 \bigr]} + \sqrt{s_n^{-d} n^{-2}\E \bigl[ |\xi_{j,n} -\E[ \xi_{j,n}]|^4 \bigr]}   \right),
\end{align*}
where we have applied \eqref{T:sparse_gamma} for the second inequality. 

Now we have by \eqref{e:bound.xijn} that $\E\bigl[|\xi_{j,n} - \E[\xi_{j,n}]|^p \bigr] \leq C^*$ for $p=3, 4$, 
so that 
$$
s_n^{-d} n^{-p/2} \E \bigl[ |\xi_{j,n} - \E [\xi_{j,n}]|^p \bigr] \leq C^* n^{1-p/2} \to 0, \ \ \ n\to\infty. 
$$
From the argument thus far we conclude that 
$$
\sum_{j \in V_n} \xi_{j,n}' \Rightarrow Z,
$$ 
which in turn implies 
$$
n^{-1/2} \bigl( \beta_{k,n,A}^{(M)}(t_i) - \E \bigl[ \beta_{k,n,A}^{(M)}(t_i) \bigr], \, i=1,\dots,m \bigr) \Rightarrow \mathcal N \bigl( 0, (\Phi_{k,A}^{(M)}(t_i, t_j))_{i,j=1}^m \bigr) 
$$
for all bounded sets $A$. 
The case when $A$ is unbounded can be established by standard approximation arguments nearly identical to those in  \cite{km} and \cite{penr}, so we omit the details and conclude that as $n\to\infty$ 
\begin{equation*} 
n^{-1/2} \bigl( \beta_{k,n}^{(M)}(t_i) - \E \bigl[ \beta_{k,n}^{(M)}(t_i) \bigr], \, i=1,\dots,m \bigr) \Rightarrow \mathcal N \bigl( 0, (\Phi_{k,\bbr^d}^{(M)}(t_i,t_j))_{i,j=1}^m \bigr).
\end{equation*}
This is equivalent to
$$
n^{-1/2} \bigl( \beta_{k,n}^{(M)}(t_i) - \E \bigl[ \beta_{k,n}^{(M)}(t_i) \bigr], \, i=1,\dots,m \bigr) \Rightarrow \bigl( \Hk^{(M)}(t_i), \, i=1,\dots,m \bigr),
$$
as $n\to\infty$. Additionally, as $M\to\infty$
$$
\big( \mathcal H_k^{(M)} (t_i), \, i=1,\dots,m \big) \Rightarrow \big( \mathcal H_k (t_i), \, i=1,\dots,m \big), 
$$
since $\Phi_{k,\bbr^d}^{(M)}(t_i, t_j) \to \Phi_{k,\bbr^d} (t_i, t_j)$ as $M\to\infty$. According to Theorem 3.2 in \cite{billingsley:1999} it suffices to show that for every $t>0$ and $\epsilon>0$, 
\begin{equation}  \label{e:billingsley.theorem3.2}
\lim_{M\to\infty}\limsup_{n\to\infty} \P \Big( \, \big| \beta_{k,n}(t) - \beta_{k,n}^{(M)}(t) -\E [\beta_{k,n}(t) - \beta_{k,n}^{(M)}(t)] \big| > \epsilon n^{1/2}  \Big) = 0. 
\end{equation}
By the Cauchy-Schwarz inequality, the probability in \eqref{e:billingsley.theorem3.2} is bounded by 
$$
\frac{1}{\epsilon^2 n} \text{Var} \big( \beta_{k,n}(t) - \beta_{k,n}^{(M)}(t) \big), 
$$
which itself converges to 
\begin{equation}  \label{e:remaining.variance}
\frac{1}{\epsilon^2}\, \sum_{i_1=M+1}^\infty \sum_{i_2=M+1}^\infty \sum_{j_1 > 0} \sum_{j_2 > 0} j_1 j_2\, \bigg( \frac{\etakrd^{(i_1, j_1, j_2)}(t_1, t_2) \delta_{i_1, i_2}}{i_1!}+ \frac{\nukrd^{(i_1, i_2, j_1, j_2)}(t_1, t_2)}{i_1! i_2!}\bigg) \ \ \text{as }n\to\infty. 
\end{equation}
Since $\Phi_{k,\bbr^d}(t,t)$ is a finite constant, \eqref{e:remaining.variance} goes to $0$ as $M\to\infty$. 
\end{proof}

\begin{proof}[Proof of Corollary \ref{c:SLLN}]
Theorem 4.6 in \cite{yogesh17} verified that 
$$
\lim_{n\to\infty} n^{-1} \big( \beta_{k,n}(t) - \E [\beta_{k,n}(t)] \big) = 0
$$
almost surely. Combining this with Proposition \ref{P:covcritical}~$(ii)$ proves the claim. 
\end{proof}

\subsection{Sparse regime}  \label{s:proof.subcritical}
As with the critical regime case, the key results for proving a central limit theorem are those on asymptotic moments that can be seen in the proposition below. As discussed in Section~\ref{s:subcritical}, the probabilistic features of these moments are asymptotically determined by $S_{k,n}(t)$. Many functions and objects in Section \ref{s:proof.critical} will be carried over for use in this section. 
\begin{proposition}\label{P:covsparse}
Let $f$ be an almost everywhere bounded and continuous density function. If $ns_n^d \to 0$ and $A \subset \mathbb{R}^d$ is open with $m(\partial A) = 0$, then we have that for $t > 0$, 
$$
\rho_n^{-1} \E [\beta_{k,n,A}(t)] \to \mu_{k,A}(t,t), \ \ \ n\to\infty,
$$
and for $t_1, t_2 > 0$, 
\begin{equation*} 
\rho_n^{-1} \Cov(\beta_{k,n,A}(t_1), \beta_{k,n,A}(t_2)) \to \mu_{k,A}(t_1,t_2), \ \ \ n\to\infty,
\end{equation*}
where 
$$
\mu_{k,A} (t_1, t_2) := \frac{1}{(k+2)!} \int_{A} f(x)^{k+2} \ \dif{x} \int_{\bbr^{d(k+1)}} h_{t_1}(0, \by)h_{t_2}(0, \by) \dif{\by}.
$$
\end{proposition}

\begin{proof}
We only discuss the covariance result in the case $A=\bbr^d$. Throughout the proof we assume $0 < t_1 \leq t_2$. 
We first derive the same expression as in \eqref{e:covariance.into.two} : 
$$
\text{Cov} \big(\beta_{k,n}(t_1), \beta_{k,n}(t_2)\big) = A_{1,n} + A_{2,n},
$$
where $A_{1,n}$ and $A_{2,n}$ are given in \eqref{e:A1n}, \eqref{e:A2n} respectively. Observing that $g_{r_n(t)}^{(k+2,j)}(\X i, \X i \cup \Pn) = 0$ for all $j \geq 2$ and any $t > 0$, we can split $A_{1,n}$ into two parts, $A_{1,n} = D_{1,n} + D_{2,n}$, where 
$$
D_{1, n} := \frac{n^{k+2}}{(k+2)!}\, \E \big[ g_{r_n(t_1)} (\X{k+2}, \X{k+2} \cup \Pn)\,  g_{r_n(t_2)} (\X{k+2}, \X{k+2} \cup \Pn) \big], 
$$
$$
D_{2,n} := A_{1,n} - D_{1,n} = \sum_{i=k+3}^\infty \sum_{j_1>0} \sum_{j_2>0} j_1j_2\frac{n^i}{i!}\, \E \bigl[ g_{r_n(t_1)}^{(i, j_1)} (\mathcal X_i, \mathcal X_i \cup \Pn) g_{r_n(t_2)}^{(i, j_2)} (\mathcal X_i, \mathcal X_i \cup \Pn)  \bigr],
$$
Based on this decomposition, we claim that
\begin{equation}  \label{e:D1n}
\rho_n^{-1} D_{1,n} \to \mu_{k,\bbr^d}(t_1,t_2),  \ \ \ n\to\infty, 
\end{equation}  
and $\rho_n^{-1} D_{2,n}$ and $\rho_n^{-1} A_{2,n}$ both converge to $0$ as $n\to\infty$. An important implication of these convergence results is that
$$
\rho_n^{-1} \text{Cov} \bigl( S_{k,n}(t_1), S_{k,n}(t_2) \bigr) \to \mu_{k,\bbr^d}(t_1,t_2), \ \ \ n\to\infty; 
$$
namely, the covariance of $\beta_{k,n}(t)$ asymptotically coincides with that of $S_{k,n}(t)$. 

By what should now be a familiar argument and the customary change of variable, we see that
\begin{align}
\rho_n^{-1} D_{1,n} &= \frac{\rho_n^{-1} n^{k+2}}{(k+2)!}  \E \bigl[ h_{r_n(t_1)} (\X{k+2}) h_{r_n(t_2)}(\X{k+2})\, \E [J_{k+2, r_n(t_2)} (\X{k+2}, \X{k+2} \cup \Pn) \bigl| \X{k+2} ] \bigr]  \label{e:rhon.D1n}  \\
&= \frac{\rho_n^{-1} n^{k+2}}{(k+2)!} \int_{\mathbb{R}^{d(k+2)}} h_{r_n(t_1)}(\bx)h_{r_n(t_2)}(\bx) \exp\bigl(-nI_{r_n(t_2)}(\bx)\bigr) \prod_{j=1}^{k+2} f(x_j) \dif{\bx} \notag\\
&= \frac{1}{(k+2)!} \int_{\mathbb{R}^{d(k+1)}} \int_{\bbr^d} h_{t_1}(0,\by) h_{t_2}(0,\by) \exp \bigl(  -nI_{r_n(t_2)}(x, x+s_n\by)\bigl) \notag \\
&\qquad \qquad \qquad \qquad \qquad \qquad \qquad  \times f(x)\prod_{j=1}^{k+1} f(x+s_ny_j) \dif{x}  \dif{\by}. \notag
\end{align}
By the continuity of $f$ it holds that $\prod_{j=1}^{k+1} f(x+s_ny_j) \to f(x)^{k+1}$ a.e.~as $n\to \infty$. 
Moreover, the exponential term converges to $1$ because we see that
$$
nI_{r_n(t_2)}(x, x+s_n \by) \leq ns_n^d \| f \|_\infty m \Bigl( \B \big(  \{ 0,\by \}; t_2 \big)  \Bigr) \to 0, \ \ n\to\infty.
$$
Thus \eqref{e:D1n} follows from the dominated convergence theorem. 

Next let us turn to the asymptotics of $\rho_n^{-1}D_{2,n}$. Proceeding as in \eqref{e:A1n.bound}, while applying \eqref{e:obvbound} and \eqref{e:spanning.tree}, we have that 
\begin{align*}
\rho_n^{-1} D_{2,n} &\leq \sum_{i=k+3}^\infty \sum_{j_1>0} \sum_{j_2>0}j_1 j_2\,  \frac{\rho_n^{-1} n^{i}}{i!}\, \E\bigg[ \one \bigl\{  \C (\mathcal{X}_i, r_n(t_1)) \text{ is connected} \bigr\} \prod_{\ell=1}^{2} b_{j_{\ell}, r_n(t_{\ell})}(\mathcal{X}_i) \bigg] \\
&\leq \sum_{i = k+3}^{\infty} \binom{i}{k+1}^2 \frac{\rho_n^{-1}n^{i}}{i!}\, \P\big(  \C (\mathcal X_{i}, r_n(t_1)) \text{ is connected} \big) \\
&\leq \frac{\big( t_1^d \| f \|_\infty \theta_d \big)^{k+1}}{\big((k+1)!\big)^2} \sum_{i=k+3}^\infty b_{i,n},  
\end{align*}
where 
$$
b_{i,n} := \frac{i! i^{i-2}}{\big( (i-k-1)! \big)^2}\, \big( nr_n(t_1)^d \| f \|_\infty \theta_d \big)^{i-(k+2)}. 
$$
Obviously $b_{i,n} \to 0$, $n\to\infty$ for all $i \geq k+3$. Since $ns_n^d \to 0$, it is easy to find a summable upper bound $c_i \geq b_{i,n}$ for sufficiently large $n$. Now the dominated convergence theorem for sums concludes $\rho_n^{-1} D_{2,n} \to 0$ as $n\to\infty$. 

For the evaluation of $n^{-1}|A_{2,n}|$, we apply \eqref{e:spanning.tree.bound2} to the right hand side at \eqref{e:A2n.bound}. Slightly changing the description of the resulting bound, we obtain
\begin{align*}
\rho_n^{-1} |A_{2,n}| &\leq 3 \cdot 2^d\, \frac{\bigl( t_2^d \| f \|_\infty \theta_d \bigr)^{k+1}}{\bigl(  (k+1)!\bigr)^2}\, \\
&\quad \times \sum_{i_1=k+2}^\infty \sum_{i_2=k+2}^\infty \frac{i_1^{i_1-1} i_2^{i_2-1}}{(i_1-k-1)! (i_2-k-1)!}\, \bigl(  nr_n(t_2)^d \| f \|_\infty \theta_d \bigr)^{i_1+i_2 - (k+2)}.  
\end{align*}
Since $ns_n^d \to 0$ as $n\to\infty$, it follows from the dominated convergence theorem for sums that $\rho_n^{-1} A_{2,n} \to 0$, $n\to\infty$, as desired. 
\end{proof}


\begin{proof}[Proof of Theorem \ref{T:sparse}]
We first establish the central limit theorem for $S_{k,n}(t)$ by proceeding in an almost identical fashion to Theorem~\ref{T:cclt}. As in that proof, we require that the left-most point of each subset $\Y \subset \Pn$ to lie in an (open) bounded set $A \subset \R^d$, with $m(\partial A) = 0$. Let $V_n, Q_{j,n}$ and $t$ be defined as in the proof of Theorem~\ref{T:cclt}. In this case however, we let $V_n$ be the vertex set of a dependency graph by letting $j \sim \pr j$ if and only if $d(Q_{j,n}, Q_{\pr j, n}) \leq 2(k+2)r_n(t)$. We modify $\xi_{j,n}$ to be defined as
\[
\xi_{j,n} := \sum_{i=1}^m a_i \sum_{\Y \subset \Pn} g_{r_n(t_i), A \cap Q_{j,n}}(\Y, \Pn)
\]
so that $\sum_{i=1}^m a_i S_{k,n,A}(t_i) = \sum_{j \in V_n} \xi_{j,n}$. 
Furthermore, $Y_{j,n}$ denotes the number of points of $\Pn$ in $\text{Tube} (Q_{j,n}, (k+2)r_n(t))$. Then, 
$$
|\xi_{j,n}| \leq \sum_{i=1}^m |a_i| \binom{Y_{j,n}}{k+2}. 
$$
It is easy to demonstrate that the Poisson parameter of $Y_{j,n}$ is bounded by $cns_n^d$ for some constant $c>0$---see \eqref{e:poi.param.bound}. Letting $C^*$ be a general positive constant as in the proof of Theorem \ref{T:cclt}, we get that for $\alpha \in \bbn$, 
$$
\E [| \xi_{j,n} |^\alpha] \leq C^* (ns_n^d)^{k+2}. 
$$
This in turn implies $\E \big[ |\xi_{j,n} - \E [\xi_{j,n}]|^p \big] \leq C^* (ns_n^d)^{k+2}$ for $p=3,4$. Let 
\[
\pr{\xi}_{j,n} := \frac{\xi_{j,n} - \E[\xi_{j,n}]}{\sqrt{\Var \big(\sum_{i=1}^m a_i S_{k,n,A} (t_i)\big)}}
\]
and $Z \sim \mathcal N(0,1)$. 
As in the critical regime case, Stein's normal approximation method gives 
\begin{align*}
\Bigl|\,  &\P\bigl( \sum_{j \in V_n} \xi_{j,n}' \leq x \bigr) - \P(Z \leq x) \, \Bigr|  \\
&\leq C^*\left( \sqrt{s_n^{-d}\rho_n^{-3/2} \E \bigl[ |\xi_{j,n}-\E[\xi_{j,n}]|^3 \bigr]} + \sqrt{s_n^{-d} \rho_n^{-2}\E \bigl[ |\xi_{j,n} -\E[ \xi_{j,n}]|^4 \bigr]}   \right),
\end{align*}
The right-hand side vanishes as $n\to\infty$, since for $p=3,4$, 
$$
s_n^{-d}\rho_n^{-p/2} \E \bigl[ |\xi_{j,n}-\E[\xi_{j,n}]|^p \bigr] \leq C^* \rho_n^{1-p/2} \to 0, \ \ \ n\to\infty.  
$$
Thus we have obtained 
\begin{equation}  \label{e:clt.Skn}
\rho_n^{-1/2} \bigl( S_{k,n}(t_i) - \E \bigl[ S_{k,n}(t_i) \bigr], \, i=1,\dots,m \bigr) \Rightarrow \mathcal N \bigl( 0, (\mu_{k,\bbr^d}(t_i,t_j))_{i,j=1}^m \bigr). 
\end{equation}
The limiting covariance matrix above coincides with the covariance functions of the process $\mathcal G_k$, i.e., 
$$
\E \bigl[ \mathcal G_k(t_i) \mathcal G_k(t_j) \bigr] = C_{f,k} \int_{\bbr^{d(k+1)}} h_{t_i} (0,\by) h_{t_j}(0,\by)   \dif \by = \mu_{k,\bbr^d}(t_i,t_j), \ \ i,j = 1,\dots,m. 
$$
Therefore \eqref{e:clt.Skn} is equivalent to
$$
\rho_n^{-1/2} \bigl( S_{k,n}(t_i) - \E \bigl[ S_{k,n}(t_i) \bigr], \, i=1,\dots,m \bigr) \Rightarrow \bigl( \mathcal G_k(t_i), \, i=1,\dots,m \bigr). 
$$
Now we can finish the entire proof, provided that for every $t>0$, 
$$
\rho_n^{-1/2} \bigl( \beta_{k,n}(t) - \E \bigl[ \beta_{k,n}(t) \bigr] \bigr)- \rho_n^{-1/2} \bigl( S_{k,n}(t) - \E \bigl[ S_{k,n}(t) \bigr] \bigr) \stackrel{p}{\to} 0, \ \ n\to\infty. 
$$
This can be proved immediately by the Cauchy-Schwarz inequality. That is, for every $\epsilon>0$, 
$$
\P \Bigl( \rho_n^{-1/2} \bigl| R_{k,n}(t) - \E [R_{k,n}(t)]  \bigr| > \epsilon \Bigr) \leq \frac{1}{\epsilon^2 \rho_n}\, \Var \bigl( R_{k,n}(t) \bigr) \to 0,
$$
where the convergence is a direct consequence of $\rho_n^{-1}D_{2,n} \to 0$ and $\rho_n^{-1}A_{2,n} \to 0$, which were verified in the proof of Proposition \ref{P:covsparse}. 
\end{proof}

\subsection{Poisson regime}  \label{s:proof.Poisson}
\begin{proof}[Proof of Theorem \ref{T:Poisson}]
We begin by defining 
$$
H_{k,n}(t) := \sum_{\Y \subset \Pn} h_{r_n(t)}(\Y),
$$
and show that 
\begin{equation}  \label{e:conv.Hkn}
\bigl( H_{k,n}(t_i), \, i=1,\dots,m \bigr) \Rightarrow \bigl( \mathcal V_k(t_i), \, i=1,\dots,m \bigr). 
\end{equation}
Subsequently we shall verify that for every $t>0$, 
\begin{align}
H_{k,n}(t) - S_{k,n}(t) &\stackrel{p}{\to} 0, \label{e:HknSkn}\\
\beta_{k,n}(t) - S_{k,n}(t) &\stackrel{p}{\to} 0.  \label{e:betaknSkn}
\end{align}
Then the proof of \eqref{e:Poisson.beta} will be complete. 
\vspace{5pt}

\noindent \emph{\underline{Part 1}}: For the proof of \eqref{e:conv.Hkn}, it is sufficient to show that for any $a_1, a_2, \dots, a_m > 0$, $m \geq 1$, 
\begin{equation*}
\sum_{i=1}^{m} a_i H_{k,n}(t_i) \Rightarrow \sum_{i=1}^{m} a_i \mathcal{V}_k(t_i).
\end{equation*}
We may use positive constants because of the fact that the Laplace transform characterizes a random vector with values in $\R_+^m$. We proceed by using Theorem 3.1 from \cite{dfsch16}. First let $(\Omega, \mathcal F, \P)$ denote a generic probability space on which all objects are defined. 
Let $\mathbf N (\bbr_+)$ be the set of finite counting measures on $\bbr_+$. We equip $\mathbf N (\bbr_+)$ with the vague topology; see, e.g., \cite{resnick:1987} for more information on the vague topology. . 
Let us define a point process $\xi_n: \Omega \to \mathbf{N}(\R_+)$ by
\[
\xi_n(\cdot) := \sum_{\Y \subset \Pn}\one \bigl\{ \sum_{i=1}^m a_i h_{r_n(t_i)} (\Y) >0  \bigr\}\,  \delta_{\sum_{i=1}^m a_i h_{r_n(t_i)}(\Y)}(\cdot), 
\]
where $\delta$ is a Dirac measure. 

Additionally let $\zeta: \Omega \to \mathbf N(\bbr_+)$ denote a Poisson random measure with mean measure $C_{f,k} \tau_k$ where 
$$
\tau_k(A) := m_k \Bigl\{ \by \in \bbr^{d(k+1)}: \sum_{i=1}^m a_i h_{t_i} (0,\by) \in A \setminus \{ 0 \}  \Bigr\}, \ \ A \subset \bbr_+. 
$$
The rest of Part 1 is devoted to  showing that
\begin{equation}  \label{e:PP.conv}
\xi_n \Rightarrow \zeta \ \ \text{in } \mathbf N(\bbr_+). 
\end{equation}
According to Theorem 3.1 in \cite{dfsch16}, the following two conditions suffice for \eqref{e:PP.conv}. Let $\mathbf L_n(\cdot) := \E [\xi_n(\cdot)]$ and $\mathbf M (\cdot) := \E[\zeta (\cdot)] = C_{f,k}\tau_k(\cdot)$. The first requirement for \eqref{e:PP.conv} is the convergence in terms of the \textit{total variation distance}:
\begin{equation}  \label{e:first.cond}
d_{\text{TV}} (\mathbf L_n, \mathbf M) := \sup_{A \in \B(\bbr_+)} \bigl| \mathbf L_n(A)  - \mathbf M(A) \bigr| \to 0, \ \ \ n \to \infty,
\end{equation} 
where $\B(\bbr_+)$ is the Borel $\sigma$-field over $\bbr_+$. 
In addition, the second requirement for \eqref{e:PP.conv} is
\begin{align}
v_n := \max_{1 \leq \ell \leq k+1} &\int_{\bbr^{d\ell}} \biggl( \int_{\bbr^{d(k+2-\ell)}} \one \Bigl\{ \sum_{i=1}^m a_i h_{r_n(t_i)} (x_1,\dots,x_{k+2}) > 0 \Bigr\} \label{e:second.cond}\\
&\qquad \qquad  \lambda^{k+2-\ell} \bigl(  \dif \, (x_{\ell+1}, \dots, x_{k+2})\bigr) \biggr)^2 \lambda^\ell \bigl(  \dif \, (x_1,\dots,x_\ell)\bigr) \to 0 \notag
\end{align}
as $n\to\infty$, where $\lambda^m = \lambda \otimes \cdots \otimes \lambda$ is a product measure on $\bbr^m$ with $\lambda(\cdot) = n\int_\cdot f(z) \dif z$. 

Let us now return to \eqref{e:first.cond} and present its proof here. 
Let $t := \max\{ t_1,\dots,t_m \}  =t_m$. Then,  for any $A \in \B(\bbr_+)$ we have from Palm theory, the change of variables $x_1=x$, $x_i=x+s_n y_{i-1}$ for $i=2,\dots,k+2$, and $\rho_n =  1$ that 
\begin{align*}
\mathbf L_n(A) &= \frac{n^{k+2}}{(k+2)!}\, \int_{\bbr^{d(k+2)}} \one \bigl\{ \, \sum_{i=1}^m a_i h_{r_n(t_i)}(\bx) \in A \setminus \{ 0 \} \bigr\} \prod_{j=1}^{k+2}f(x_j) \dif \bx \\
&= \frac{1}{(k+2)!}\, \int_{\bbr^{d(k+2)}}  \one \bigl\{ \, \sum_{i=1}^m a_i h_{t_i}(0, \by) \in A \setminus \{ 0 \} \bigr\} f(x) \prod_{j=1}^{k+1} f(x+s_n y_j) \dif x \dif \by. 
\end{align*}
Therefore, 
\begin{align*}
&\bigl| \mathbf L_n(A)  - \mathbf M(A)  \bigr| \\
&\leq \frac{1}{(k+2)!}\, \int_{\bbr^{d(k+2)}}  \one \bigl\{ \, \sum_{i=1}^m a_i h_{t_i}(0, \by) \in A \setminus \{ 0 \} \bigr\} f(x)\, \Bigl| \prod_{j=1}^{k+1} f(x+s_n y_j)  -f(x)^{k+1}  \Bigr|  \dif x \dif \by. 
\end{align*} 
If the indicator function above is equal to $1$, then $h_{t_i}(0,\by) = 1$ for at least one $i$, which means that the distance of each component in $\by$ from the origin must be less than $t$. Otherwise one cannot form a required empty $(k+1)$-simplex. 
Hence we have 
\begin{align*}
\bigl| \mathbf L_n(A)  - \mathbf M(A)  \bigr|  \leq \frac{1}{(k+2)!}\, \int_{\bbr^{d(k+2)}} \prod_{i=1}^{k+2} \one \{ |y_i| \leq t \} f(x)\, \Bigl| \prod_{j=1}^{k+1} f(x+s_n y_j)  -f(x)^{k+1}  \Bigr|  \dif x \dif \by. 
\end{align*}
We have by continuity of $f$ that $\bigl| \prod_{j=1}^{k+1} f(x+s_n y_j)  -f(x)^{k+1}  \bigr|$ converges to $0$ a.e. ~as $n\to\infty$ and is bounded by $2\|f \|_{\infty}^{k+1} < \infty$. So the dominated convergence theorem applies to get $\bigl| \mathbf L_n(A) - \mathbf M(A) \bigr| \to 0$ as $n\to\infty$. Since this convergence holds uniformly for all $A\in \B(\bbr_+)$, we have now established \eqref{e:first.cond}. 

Next we turn to proving \eqref{e:second.cond}. First we can immediately see that 
\begin{align*}
v_n = \max_{1 \leq \ell \leq k+1} &n^{2k+4-\ell} \int_{\bbr^{d(2k+4-\ell)}} \one \Bigl\{ \sum_{i=1}^m a_i h_{r_n(t_i)} (x_1,\dots,x_{k+2}) > 0 \Bigr\} \\
&\times \one \Bigl\{ \sum_{i=1}^m a_i h_{r_n(t_i)} (x_1,\dots,x_\ell, x_{k+3}, \dots, x_{2k+4-\ell}) > 0 \Bigr\} \prod_{j=1}^{2k+4-\ell} f(x_j) \dif \bx. 
\end{align*}
Making a change of variables with $x_1 = x$ and $x_i = x + s_ny_{i-1}$ for $i=2,\dots,2k+4-\ell$, while using $f(x+s_ny_{i-1}) \leq \|  f\|_\infty$, we get that 
\begin{align*}
v_n \leq \| f \|_\infty^{2k+3-\ell} \max_{1 \leq \ell \leq k+1} &n^{2k+4-\ell}s_n^{d(2k+3-\ell)}  \int_{\bbr^{d(2k+3-\ell)}}  \one \Bigl\{ \sum_{i=1}^m a_i h_{t_i} (0,y_1,\dots,y_{k+1}) > 0 \Bigr\} \\
&\qquad \times \one \Bigl\{ \sum_{i=1}^m a_i h_{t_i} (0,y_1, \dots,y_{\ell-1}, y_{k+2}, \dots, y_{2k+3-\ell}) > 0 \Bigr\}\dif \by. 
\end{align*}
Obviously the above integral is finite, and 
$$
\max_{1 \leq \ell \leq k+1} n^{2k+4-\ell}s_n^{d(2k+3-\ell)} = \max_{1 \leq \ell \leq k+1} (ns_n^d)^{k+2-\ell} \to 0, \ \ \ n\to\infty,
$$
by the assumption $\rho_n=1$. So $v_n \to 0$ follows and \eqref{e:second.cond} is obtained. 
\bigskip

\noindent \emph{\underline{Part 2}}: Define the map $\widehat T: \mathbf N (\bbr_+) \to \bbr_+$ by $\widehat T (\sum_n \delta_{x_n}) = \sum_n x_n$. This map is continuous because it is defined on the space of \textit{finite} counting measures. Applying the continuous mapping theorem to \eqref{e:PP.conv} gives $\widehat T(\xi_n) \Rightarrow \widehat T (\zeta)$. Equivalently, we have
$$
\sum_{i=1}^m a_i H_{k,n}(t_i)\Rightarrow \sum_{i=1}^m a_i \mathcal V_k (t_i). 
$$
To see such equivalence, note that $\widehat T(\xi_n) = \sum_{i=1}^m a_i H_{k,n}(t_i)$, so it now suffices to  show that $\widehat T(\zeta)$ is equal in distribution to $\sum_{i=1}^m a_i \mathcal V_k(t_i)$. 
To this aim let us represent $\zeta$ as 
$$
\zeta \stackrel{d}{=} \sum_{i=1}^{M_n} \delta_{Y_i}, 
$$
where $Y_1,Y_2,\dots$ are i.i.d with common distribution $\tau_k(\cdot)/\tau_k (\bbr_+)$ and $M_n$ is Poisson distributed with parameter $C_{f,k}\tau_k (\bbr_+)$. Further, $(Y_i)$ and $M_n$ are independent. On one hand, it follows from the Laplace functional of a Poisson random measure (see Theorem 5.1 in \cite{resnick:2007}) that for every $\lambda >0$,  
\begin{align*}
\E \Bigl[ \exp \bigl( -\lambda \sum_{i=1}^m a_i \mathcal V_k(t_i) \bigr)  \Bigr] &= \E \biggl[  \exp\Big( -\int_{\bbr^{d(k+1)}} \lambda \sum_{i=1}^m a_i h_{t_i} (0,\by) M_k(\dif \by) \Big) \biggr] \notag \\
&= \exp \biggl(  - C_{f,k} \int_{\bbr^{d(k+1)}} \bigl( 1-e^{-\lambda \sum_{i=1}^m a_i h_{t_i} (0,\by)} \bigr) \dif \by \biggr)
\end{align*}
On the other hand it is straightforward to compute that 
\begin{align*}
\E \bigl[ \exp\big(-\lambda \widehat T (\zeta) \big)  \bigr] &= \E \Bigl[  \exp\Big( -\lambda \sum_{i=1}^{M_n} Y_i \Big) \Bigr]
= \exp \Bigl( -C_{f,k}\tau_k(\bbr_+) (1-\E[e^{-\lambda Y_1}]) \Bigr) \\
&= \exp \Bigl(  - C_{f,k} \int_{\bbr^{d(k+1)}} \bigl( 1-e^{-\lambda \sum_{i=1}^m a_i h_{t_i} (0,\by)} \bigr) \dif \by \Bigr), 
\end{align*}
implying $\widehat T(\zeta) \stackrel{d}{=} \sum_{i=1}^m a_i \mathcal V_k(t_i)$ as required.  
\bigskip

\noindent \emph{\underline{Part 3}}: It remains to show \eqref{e:HknSkn} and \eqref{e:betaknSkn}. As for \eqref{e:HknSkn}, we know from \eqref{e:D1n} with $\rho_n = 1$ and $t_1 = t_2$, that 
$$
\E [S_{k,n}(t)] \to \mu_{k,\bbr^d}(t,t),  \ \ \ n\to\infty. 
$$
Since the exponential term in \eqref{e:rhon.D1n} converges to $1$ without affecting the value of the limit, it must be that the $\E [H_{k,n}(t)]$ and $\E[S_{k,n}(t)]$ have the same limit. That is,  
$$
\E [H_{k,n}(t)] \to \mu_{k,\bbr^d}(t,t), \ \ \ n\to\infty,
$$
and thus, the Markov inequality gives \eqref{e:HknSkn}. 

Finally we turn our attention to \eqref{e:betaknSkn}. By Markov's inequality, it suffices to show that $\E [R_{k,n}(t)] \to 0$ as $n\to\infty$. Mimicking the derivation of \eqref{e:bound.A1n} with $\rho_n = 1$, we get that 
$$
\E[R_{k,n}(t)] \leq \sum_{i=k+3}^\infty \binom{i}{k+1} \frac{n^i}{i!}\, \P \bigl( \C (\mathcal X_i, r_n(t)) \text{ is connected} \bigr). 
$$
Recalling the bound in \eqref{e:spanning.tree}, we  have 
\begin{align*}
\E [R_{k,n}(t)] &\leq \frac{\bigl( t^d \| f \|_\infty \theta_d \bigr)^{k+1}}{(k+1)!}\, \sum_{i=k+3}^\infty \frac{i^{i-2}}{(i-k-1)!}\, \bigl( nr_n(t)^d \| f \|_\infty \theta_d \bigr)^{i-(k+2)} \to 0
\end{align*}
as $n\to\infty$. 
\end{proof} 

\bibliography{ProcesslevelBetti}

\end{document}